\documentclass{article}[11pt]

\usepackage{amsmath}   
\usepackage{amssymb}
\usepackage{amsthm}
\usepackage{accents}

\newtheorem{Th}{Theorem}[section] 
\newtheorem{Prop}{Proposition}[section]   
\newtheorem{Lem}{Lemma}[section]   
\newtheorem{Coro}{Corollary}[section]   
  
\newtheorem{Rem}{Remark}[section]

\newcommand{\finishproof}{\hfill $\Box$ \vspace{1 mm}}
\newcommand{\R}{\mathbb{R}}
\newcommand{\Z}{\mathbb{Z}}
\newcommand{\C}{\mathbb{C}}
\newcommand{\T}{\mathbb{T}}

\newcommand{\s}{{\rm S}}
\newcommand{\x}{\langle x\rangle}
\newcommand{\A}{{\mathcal A}}
\newcommand{\AH}{{\mathbb A}}

\newcommand{\id}{\text{\rm id}}
\newcommand{\Id}{\text{\rm Id}}

\newcommand{\const}{\text{\rm const}}
\newcommand{\tr}{\text{\rm tr}\,}

\newcommand{\Div}{\mathop{\rm div}\nolimits}

\newcommand{\dt}[1]{\accentset{\mbox{\bfseries .}}{#1}}
\newcommand{\cdt}{\mbox{\bfseries .}}

\begin{document}

\title{Spatial asymptotic expansions in the incompressible Euler equation}   
 
\author{R. McOwen and P. Topalov} 

\maketitle

\begin{abstract}  
In this paper we prove that the Euler equation describing the motion of an ideal fluid in $\R^d$
is well-posed in a class of functions allowing spatial asymptotic expansions as $|x|\to\infty$
of any a priori given order. These asymptotic expansions can involve log terms and lead to a family of 
conservation laws. Typically, the solutions of the Euler equation with rapidly decaying initial data develop non-trivial 
spatial asymptotic expansions of the type considered here.
\end{abstract}   

%\tableofcontents 

%%%%%%%%%%%%%%%%%%%%%%%%%%%%%%%%%
\section{Introduction}\label{sec:introduction}
Consider the Euler equation for the motion of an incompressible perfect fluid in $\R^d$ for $d\ge 2$:
\begin{equation}\label{eq:euler}
\left\{
\begin{array}{l}
u_t+u\cdot \nabla u=-\nabla{\rm p},\quad \Div u =0,\\
u|_{t=0}=u_0,
\end{array}
\right.
\end{equation}
where the velocity field $u(t,x)$  is to be determined from the given initial condition $u_0(x)$.
Here ${\rm p}(t,x)$ is the scalar pressure; the divergence $\Div$ and the covariant derivative $\nabla$ are computed 
with respect to the Euclidean metric in $\R^d$ and the expression $u\cdot \nabla u$ stands for the
directional covariant derivative $\nabla_u u$.
Throughout the paper we identify vectors and covectors with the help of the Euclidean metric.
A classical result of Kato \cite{Kato2} states that if the initial data $u_0$ belongs to the Sobolev space $H^m$, 
$m>1+\frac{d}{2}$, of vector fields in $\R^d$ whose derivatives up to order $m$ are in $L^2$, then there exists 
$T>0$ so that equation \eqref{eq:euler} has a unique solution 
$u\in C^0\big([0,T],H^m\big)\cap C^1\big([0,T], H^{m-2}\big)$. 
Note that the condition $u_0\in H^m$ for $m>d/2$ implies that $u_0(x)\to 0$ as $|x|\to\infty$ (and similarly 
for $u(t,x)$); this leaves out, for example, the ``physically interesting case of flows having uniform velocity at infinity'' 
\cite{Kato1}.

\medskip

In this paper, we consider spaces of vector fields that allow {\em asymptotic expansions} at infinity of any a priori 
given order. Note that these vector fields do not necessarily decay at infinity. The simplest form of such asymptotic 
expansion is
\begin{equation}\label{eq:u_0-asymptotic}
a_0(\theta)+\frac{a_1(\theta)}{r}+\cdots+\frac{a_N(\theta)}{r^N}+o\Big(\frac{1}{r^N}\Big) \quad
{\rm as}\ |x|\to\infty,
\end{equation}
where $N$ is a non-negative integer, $r\equiv r(x):=|x|$, $\theta\equiv\theta(x):=\frac{x}{|x|}$ is a point on the unit 
sphere $\s^{d-1}$ in $\R^d$, 
and the coefficients $a_j$ for $j=0,...,N$ are continuous vector-valued functions $a_j : \s^{d-1}\to\R^d$.
Under certain technical assumptions on $a_1,\dots, a_N$ and the remainder in \eqref{eq:u_0-asymptotic}, 
we have introduced and studied in \cite{McOwenTopalov2} Banach spaces $\A^{m,p}_{N}$ of such vector fields that lie
inside $H^{m,p}_{loc}$ with $m$ a non-negative integer and $1<p<\infty$; 
the details of these {\em asymptotic spaces} $\A^{m,p}_{N}$ are summarized in the first section of this paper. 
Note that the spaces $\A^{m,p}_{N}$ are {\em not} preserved by the Euler flow 
(see Example 1 in Appendix \ref{sec:examples} for a counter example), so we generalize \eqref{eq:u_0-asymptotic} by
allowing logarithmic terms as follows:
\begin{equation}\label{eq:u_0-asymptotic2}
a_0(\theta)+\frac{a_1^0(\theta)+a_1^{1}(\theta)\log r}{r}+\cdots+
\frac{a_N^0(\theta)+\cdots+a_N^{N}(\theta)(\log r)^{N}}{r^N}+o\Big(\frac{1}{r^N}\Big) \quad
{\rm as}\ |x|\to\infty.
\end{equation}
In Appendix B of \cite{McOwenTopalov2}, we introduced and studied a class of Banach spaces 
$\A^{m,p}_{N;\ell}$ which for $\ell=0$ consist of vector fields having the asymptotic behavior in 
\eqref{eq:u_0-asymptotic2}; the details of the asymptotic spaces $\A^{m,p}_{N;\ell}$ are also summarized 
in the next section. One of the main results of this paper implies that the asymptotic spaces $\A^{m,p}_{N;0}$ 
{\em are preserved} by the Euler flow.

\medskip

Let us now formulate our first result. Let  $B_{\A^{m,p}_{N;0}}(\rho)$
be the open ball of radius $\rho>0$ centered at the origin in $\A^{m,p}_{N;0}$.
Denote by $\accentset{\,\,\,\circ}{\A}^{m,p}_{N;0}$ the closed subspace of {\em divergence free} vector fields 
in ${\A}^{m,p}_{N;0}$.

\begin{Th}\label{th:main} 
Assume that $m>3+\frac{d}{p}$.
For any given $\rho>0$ there exists $T>0$ such that for any divergence free vector field
$u_0\in B_{\A^{m,p}_{N;0}}(\rho)$ there exists a unique solution 
$u\in C^0\big([0,T],\accentset{\,\,\,\circ}\A^{m,p}_{N;0}\big)\cap 
C^1\big([0,T],\accentset{\,\,\,\circ}\A^{m-1,p}_{N;0}\big)$ 
of the Euler equation \eqref{eq:euler} that depends continuously on the initial data in the sense that 
the data-to-solution map
\[
u_0\mapsto u,\,\,B_{\A^{m,p}_{N;0}}(\rho)\cap{\accentset{\,\,\,\circ}{\A}}^{m,p}_{N;0}\to C^0\big([0,T],
\accentset{\,\,\,\circ}\A^{m,p}_{N;0}\big)\cap C^1\big([0,T],\accentset{\,\,\,\circ}\A^{m-1,p}_{N;0}\big)
\]
is continuous. In particular, the space $\accentset{\,\,\,\circ}\A^{m,p}_{N;0}$ is preserved by the Euler flow.
\end{Th}

\begin{Rem}
For any $t\in[0,T]$ the pressure term ${\rm p}(t)$ has an asymptotic expansion as $|x|\to\infty$ of a special
form (see formula \eqref{eq:pressure} and Remark \ref{rem:laplace_operator} for the definition of the corresponding
asymptotic space).
\end{Rem}

\begin{Rem}
A simple modification of our proof shows that a variant of Theorem \ref{th:main} also holds if one allows an 
external force of the form ${\rm f}\in C^0\big([0,\tau],\A^{m+1,p}_{N;0}\big)$, $\tau>0$. 
More specifically, one can prove that for any 
divergence free vector field $q\in\accentset{\,\,\,\circ}\A^{m,p}_{N;0}$ there exist an open neighborhood
${\mathcal U}(q)$ of $q$ in $\accentset{\,\,\,\circ}\A^{m,p}_{N;0}$ and $0<T\le\tau$ so that 
for any $u_0\in{\mathcal U}(q)$ there exists a unique solution 
$u\in C^0\big([0,T],\A^{m,p}_{N;0}\big)\cap C^1\big([0,T],\A^{m-1,p}_{N;0}\big)$ 
of the Euler equation $u_t+u\cdot \nabla u=-\nabla{\rm p}+{\rm f}$, $\Div u =0$, that depends continuously on 
the initial data in the sense described in Theorem \ref{th:main}.
\end{Rem}

\begin{Rem}\label{rem:refined_main_thm}
Theorem \ref{th:main} admits a refinement to asymptotic spaces $\A^{m,p}_{n,N;0}$, $1\le n\le N$, which are defined 
by requiring that the coefficients $a_0$ and $a_k^j$ for $1\le k\le n-1$ and $0\le j\le k$ in \eqref{eq:u_0-asymptotic2}
vanish (cf.\ \eqref{eq:asymptotic_log_space} and Section \ref{sec:the_group} for more detail on this space).
In fact, the statement of Theorem \ref{th:main} holds with $\A^{m,p}_{N;0}$ replaced by $\A^{m,p}_{n,N;0}$, 
provided $0\leq n\leq d+1$. If the initial data belongs to $\A^{m,p}_{n,N;0}$ with $n> d+1$ then the solution
in Theorem \ref{th:main} evolves in the space $\A^{m,p}_{d+1,N;0}$. Moreover, such a solution typically develops
a non-vanishing leading asymptotic term $a_{d+1}/r^{d+1}$ (see Example 2 in Appendix \ref{sec:examples} as well 
as \cite{DobrShaf}).
The proof of Theorem \ref{th:main} for the spaces $\A^{m,p}_{n,N;0}$, $1\le n\le N$, does not significantly 
vary from the case $n=0$, so we will not discuss the details. We only mention that the condition $0\le n\leq d+1$ 
above appears as a consequence of Lemma \ref{lem:inverse_laplacian_remainder_space} and a generalization of
Lemma \ref{lem:vanishing_lemma}.
\end{Rem}

A few more comments about Theorem \ref{th:main} are in order. First of all, it is natural to wonder whether the 
asymptotic solutions in Theorem \ref{th:main} persist for all time in the case $d=2$ when global existence results 
are known (\cite{Wol,Kato1}): indeed we expect this to be the case. Secondly, taking $n>d/2$ in 
Remark \ref{rem:refined_main_thm}, we see that our solutions have finite kinetic energy
$\int_{\R^d} |u(x,t)|^2\,dx<\infty$. Finally, we observe that, unlike other non-linear evolution equations such as KdV, 
the modified KdV, Camassa-Holm, and Degasperis-Procesi equations, the solution map of the Euler equation 
\eqref{eq:euler} does not preserve the Schwartz class ${\mathcal S}$, or more generally, the asymptotic space
$\A^{m,p}_{n,N;0}$ with $n>d+1$ (see Remark \ref{rem:refined_main_thm}). More specifically, solutions with initial data 
in ${\mathcal S}$ evolve in the asymptotic space $\A^{m,p}_{d+1,N;0}$ and typically develop a non-vanishing leading 
asymptotic term $a_{d+1}/r^{d+1}$.
In this way, the asymptotic spaces appear {\em naturally} in the Euler dynamics of the velocity field, and 
the asymptotic space $\A^{m,p}_{d+1,N;0}$ in particular can be considered as a natural ``minimal" extension 
of the Schwartz class that stays invariant with respect to the Euler flow.

\medskip\medskip

\noindent{\em Asymptotic conserved quantities.} Take $u_0\in\A^{m,p}_{N;0}$. Then, by 
Theorem \ref{th:main}, there exist $T>0$ and a unique solution 
$u\in C^0\big([0,T],\accentset{\,\,\,\circ}\A^{m,p}_{N;0}\big)\cap 
C^1\big([0,T],\accentset{\,\,\,\circ}\A^{m-1,p}_{N;0}\big)$ 
of the Euler equation \eqref{eq:euler}. 
In particular, for any $t\in[0,T]$ we have $u(t)\in\A^{m,p}_{N;0}$ and hence $u(t)$ allows an asymptotic expansion
of the form \eqref{eq:u_0-asymptotic2}. In general, the coefficients 
\[
a_k^j(t)\in C^0\big(\s^{d-1},\R^d\big),\quad 0\le k\le N,\,\,\,\,0\le j\le k, 
\]
of this asymptotic expansion depend on $t\in[0,T]$.
We will see, however, that the leading coefficient $a_0\equiv a^0_0$ in this asymptotic expansion is 
independent of $t$. Hence,  for any $t\in[0,T]$ and for any $\omega\in\s^{d-1}$, the limit
\[
\lim_{\varrho\to\infty} u(t,\varrho\,\omega)
\]
exists and is independent of $t$. In this way, $a_0$ is a {\em conserved quantity} 
(or, equivalently, an {\em integral}) of the Euler equation.
More generally, we have additional conserved quantities:
Recall that the eigenvalues of the ``positive''  Laplace-Beltrami operator $-\Delta_S$ on the unit sphere 
$\s^{d-1}\subseteq\R^d$, equipped with the restriction of the Euclidean metric on $\R^d$ to $\s^{d-1}$,
have the form
\[
\mu_l=l(l+d-2),\quad l=0,1,2,...\,,
\]
and the space of eigenfunctions with eigenvalue $\mu_l$ coincides with the space of homogeneous harmonic
polynomials of degree $l$ in $\R^d$ restricted to $\s^{d-1}$ (see e.g.\ \cite{Shubin}). We have the following
\begin{Prop}\label{prop:integrals}
Assume that $u_0\in\accentset{\,\,\,\circ}\A^{m,p}_{N;0}$ and let  $u$ be the solution of 
the Euler equation \eqref{eq:euler} given by Theorem \ref{th:main}.
\begin{itemize}
\item[$(a)$] For any $0\le k<\max\{1,d-2\}$ the coefficient $a_k^k : \s^{d-1}\to\R^d$ is independent of $t\in[0,T]$;
\item[$(b)$] For any $\max\{1,d-2\}\le k\le N$, $l\ge 0$, $l\ne k-d+2$, and for any homogeneous harmonic polynomial 
$h$ of degree $l$ in $\R^d$,  the quantity 
\begin{equation}\label{eq:integrals}
\big(a_k^k,h|_{\s^{d-1}}\big)_{\s^{d-1}}:=\int_{\omega\in\s^{d-1}} h(\omega)\,a_k^k(t,\omega)\,{\rm d}s(\omega),
\end{equation}
where ${\rm d}s$ is the Riemannian volume form on $\s^{d-1}$, is independent of $t\in[0,T]$.
\end{itemize}
\end{Prop}

\begin{Rem}\label{rem:integrals}
The majority of the integrals \eqref{eq:integrals} of the Euler flow are non-trivial 
(see the end of Section \ref{sec:proof_main_theorem}). 
\end{Rem}

\begin{Rem}
Arguing as in the proof of Proposition \ref{prop:integrals}, one can see that if the initial data lies in $\A^{m,p}_{n,N;0}$, 
$0\le n\le N$, then the leading asymptotic term in the asymptotic expansion 
\eqref{eq:asymptotic_log_space} of the solution $u(x,t)$ is independent of $t\in[0,T]$.
Moreover, the coefficients $a_k^j$ with $n\le k\le\min\{2n,d\}$, $0\le j\le k$, in the asymptotic
expansion of the solution are independent of $t\in[0,T]$.
\end{Rem}

\medskip

\noindent{\em Asymptotic diffeomorphisms.}
Let $u\in C^0\big([0,T],\accentset{\,\,\,\circ}\A^{m,p}_{N;0}\big)\cap 
C^1\big([0,T],\accentset{\,\,\,\circ}\A^{m-1,p}_{N;0}\big)$ be the solution of the Euler equation
\eqref{eq:euler} given by Theorem \ref{th:main}. In Section \ref{sec:ode} we prove that
the one-parameter family of diffeomorphisms $x\mapsto\varphi(t,x)$, $\R^d\to\R^d$, $t\in[0,T]$,
associated to the non-autonomous dynamical system $\dot x=u(t,x)$ on $\R^d$ consists of 
volume preserving diffeomorphisms of a special form, that we call {\em asymptotic diffeomorphisms}. 
This class of diffeomorphisms forms a group that we denote by $\accentset{\,\,\,\circ}\A D^{m,p}_{N;0}$. 
This group is a real analytic Banach manifold
and its basic properties are studied in Section \ref{sec:volume_preserving_diffeomorphisms}
(see Theorem \ref{th:volume_preserving_diffeomorphisms} and Proposition \ref{prop:exp_local_diffeomorphism'}.).

\medskip

\noindent{\em Related work.} There is a vast literature on the well-posedness of the incompressible Euler equation 
(see e.g.\ \cite{BardosTiti,Kato1} and the references therein).
Here we only discuss works related to solution of fluid equations in $\R^d$ with various asymptotic or decay conditions
at infinity. In \cite{C} Cantor proved that, for initial data $u_0$ in certain weighted Sobolev spaces 
decaying at infinity, the solution of the Euler equation also has this decay at infinity. 
Similar result was proved by A. Constantin for the Camassa-Holm equation in \cite{Con2}.
In a series of papers \cite{BS1,BS2,BS3}, Bondareva and Shubin proved that the solution map of the KdV equation on 
the line preserves a class of functions allowing asymptotic expansions at infinity of infinite order. 
Bondareva and Shubin use a method developed by Menikoff \cite{Menikoff}.  Menikoff's approach is based on the
integrability of the KdV equation and can not be extended to more general equations. A similar result was proved 
for the modified Korteweg-de Vries equation in \cite{KPST}, by using the
properties of the Miura transform and the result of  Bondareva and Shubin.  Note, however, that neither Cantor nor
Menikoff considers functions having asymptotic expansions at infinity. 
It is worth mentioning that specific spatial decay and even certain asymptotic expansions are known to naturally appear 
for the solutions of the Navier-Stokes equation with {\em rapidly decreasing} initial data 
(see e.g.\ \cite{DobrShaf,BranMe,Bran,KR} and the references therein).
In dimension one, applications of the  asymptotic space $\A^{m,2}_N(\R)$ and an analogous space 
$\AH^{m,2}_N(\R)$ (with a different remainder term) to the Camassa-Holm and the Degasperis-Procesi equations 
were obtained in \cite{McOwenTopalov1}.
Finally, note also that there is a vast literature on asymptotic expansions of solutions of several classes of ordinary 
differential equations (see e.g.\ \cite{Erdelyi}).

\medskip

\noindent{\em Organization of the paper.} The paper is organized as follows.
In Section \ref{sec:the_group} we recall the definition of the asymptotic spaces of functions and the groups of 
asymptotc diffeomorphisms $\A D^{m,p}_{N;0}$ and $\accentset{\,\,\,\circ}\A D^{m,p}_{N;0}$. 
At the end of this section we formulate Proposition \ref{prop:ode} and 
Corollary \ref{coro:ode_volume_preserving}, which allows us to use Arnold's approach to fluid dynamics. 
Section \ref{sec:laplace_operator} is devoted to the properties of the
Laplace operator in asymptotic spaces. In Section \ref{sec:euler_vector_field} we define the Euler vector field
on the tangent bundle of the asymptotic group $\A D^{m,p}_{N;0}$.
The real-analyticity of this vector field is proved in Section \ref{sec:conjugate_laplace_operator} and 
Section \ref{sec:smoothness_of_the _vector_field}. Theorem \ref{th:main} and Proposition \ref{prop:integrals}
are proved in Section \ref{sec:proof_main_theorem}. In Section \ref{sec:ode} we prove
Proposition \ref{prop:ode} and Corollary \ref{coro:ode_volume_preserving} formulated in
Section \ref{sec:the_group}. Section \ref{sec:volume_preserving_diffeomorphisms} is devoted to
the differential geometry of the subgroup of volume preserving asymptotic diffeomorphisms.
In particular, we prove in Section \ref{sec:volume_preserving_diffeomorphisms} that the exponential map corresponding to
the Euler equation, when restricted to a small open neighborhood of zero, is a real analytic diffeomorphism onto its
image. In Appendix \ref{sec:appendix_properties} we collect the basic results on the asymptotic spaces that are used
in the main body of the paper.
In Appendix \ref{sec:examples} we construct two Examples that show the necessity of introducing spatial asymptotic
expansions involving log terms.

\medskip 

\noindent{\em Acknowledgment.} The authors are grateful to Sergei Kuksin and Gerard Misiolek for useful comments
on this paper.

%%%%%%%%%%%%%%%%%%%%%%%%%%%%%%%%%%
\section{Asymptotic spaces and diffeomorphism groups}\label{sec:the_group}
Here we recall the definition of the asymptotic spaces of functions and groups of asymptotic diffeomorphisms that were  
introduced in \cite{McOwenTopalov2}. For the convenience of the reader we also summarize some of the basic properties of 
these spaces in Appendix \ref{sec:appendix_properties}. 

\medskip

\noindent{\em Asymptotic spaces.}
Let $\chi(\varrho)$ be a $C^\infty$-smooth cut-off function such that $\chi(\varrho)=1$ for $\varrho\ge 2$ and 
$\chi(\varrho)=0$ for $0\le\varrho\le 1$. For $1<p<\infty$ and a nonegative integer $m$, let $ H^{m,p}_{loc}(\R^d)$
denote the space of functions whose weak derivatives up to order $m$ are locally $L^p$-integrable on $\R^d$. 
For any $\delta\in\R$ let us define  the weighted Sobolev space
\begin{equation}\label{eq:remainder_space}
W^{m,p}_\delta(\R^d):=
\big\{f\in H^{m,p}_{loc}(\R^d)\,\big|\,\x^{\delta+|\alpha|}\partial^\alpha f\in L^p,\,\forall\alpha\,\,\text{such that}\,\,
|\alpha|\le m\big\},\, 
\end{equation}
with norm
\[
\|f\|_{W^{m,p}_\delta}:=\sum_{|\alpha|\le m}\|\langle x\rangle^{\delta+|\alpha|}\partial^\alpha f\|_{L^p},
\]
where $\x:=\sqrt{1+|x|}$\,, $|x|$ is the Euclidean norm of $x\in\R^d$, $\alpha=(\alpha_1,...,\alpha_d)$ is a
multi-index with $|\alpha|=\alpha_1+...+\alpha_d$, and
$\partial^\alpha:=\partial_1^{\alpha_1}\partial_2^{\alpha_2}\cdots\partial_d^{\alpha_d}$ where $\partial_k$ 
is the weak partial derivative $\frac{\partial}{\partial x_k}$.
For $m>\frac{d}{p}$ and any integer  $N\ge 0$, we can now define the {\em asymptotic space}
$\A^{m,p}_{N}(\R^d)$ to consist of functions $u\in H^{m,p}_{loc}(\R^d)$ of the form
\begin{equation}\label{eq:asymptotic_space}
u(x)=\chi(r)\Big(a_0(\theta)+\frac{a_1(\theta)}{r}+\cdots+\frac{a_N(\theta)}{r^N}\Big)+f(x),
\end{equation}
where $r=|x|$, $a_k\in H^{m+1+N-k,p}(\s^{d-1})$ and 
the {\em remainder} function $f(x)$ belongs to the weighted Sobolev  space $W^{m,p}_{\gamma_N}(\R^d)$.
Here $\gamma_N:=N+\gamma_0$ where $\gamma_0$ is independent of $m$ and $N$ and is chosen so 
that $0<\gamma_0+d/p<1$. With $\gamma_N$ chosen this way, the remainder in \eqref{eq:asymptotic_space} 
satisfies (see Lemma \ref{lem:remainder_space} in Appendix \ref{sec:appendix_properties}),
\begin{equation}\label{eq:small_o}
f(x)=o\Big(\frac{1}{r^{N}}\Big)\,\,\,\text{as}\ |x|\to\infty\,.
\end{equation}
Hence, the elements of $\A^{m,p}_{N}(\R^d)$ satisfy the asymptotic formula
\eqref{eq:u_0-asymptotic}. When the domain $\R^d$ is understood, we simply write $\A^{m,p}_{N}$ instead of 
$\A^{m,p}_{N}(\R^d)$.

\begin{Rem}
By Lemma \ref{lem:remainder_space} in Appendix \ref{sec:appendix_properties}, \eqref{eq:small_o} will hold if 
we only require $\gamma_N\ge N-\frac{d}{p}$. The stronger condition on $\gamma_N$ is imposed so that 
$\Delta : W^{m+1,p}_{\gamma_N}\to W^{m-1,p}_{\gamma_N+2}$ has a closed image (see \cite{McOwen}).
\end{Rem}

\begin{Rem}
The condition $a_k\in H^{m+1+N-k,p}(\s^{d-1})$, $k=0,...,N$, may seem at first unnatural. In fact, 
this condition is crucial for defining  groups of asymptotic diffeomorphisms below. 
\end{Rem}

\noindent Assume that  $0\le n\le N$ and denote by $\A^{m,p}_{n,N}$ the closed subspace in 
$\A^{m,p}_{N}$ defined by the condition that the first $n$ coefficients $a_0,...,a_{n-1}$ in the asymptotic 
expansion of $u\in \A^{m,p}_{N}$ vanish. If $n\ge N+1$ we set 
$\A^{m,p}_{n,N}:=W^{m,p}_{\gamma_N}$.
The asymptotic space $\A^{m,p}_{n,N}$ is a Banach space with norm,
\begin{equation}\label{eq:the_norm}
\|u\|_{\A^{m,p}_{n,N}}:=\sum_{k=n}^{N}\|a_k\|_{ H^{m+1+N-k,p}(\s^{d-1})}+\|f\|_{W^{m,p}_{\gamma_N}}\,.
\end{equation}

Now let us discuss how to include {\em log-terms} in the asymptotic spaces. For $0\leq n \leq N$ and an integer 
$\ell\ge -n$, the space $\A^{m,p}_{n,N;\ell}(\R^d)$ is defined in the same way as 
$\A^{m,p}_{n,N}(\R^d)$ except that formula \eqref{eq:asymptotic_space} is now replaced by
\begin{equation}\label{eq:asymptotic_log_space}
u(x)=\chi(r)\Big(\frac{a_n^0+\cdots+a_n^{n+\ell}(\log r)^{n+\ell}}{r^n}+\cdots+
\frac{a_N^0+\cdots+a_N^{N+\ell}(\log r)^{N+\ell}}{r^N}\Big)+f(x),
\end{equation}
where $a_k^j\in H^{m+1+N-k}(\s^{d-1})$ for $0\le j\le k+\ell$ and $0\le n\le k\le N$, and 
$f\in W^{m,p}_{\gamma_N}(\R^d)$. In particular, the remainder $f$ satisfies \eqref{eq:small_o}.
The norm in $\A^{m,p}_{n,N;\ell}(\R^d)$ is defined  in a similar way as in \eqref{eq:the_norm},
\[
\|u\|_{\A^{m,p}_{n,N;\ell}}:=
\sum_{n\le k\le N, 0\le j\le k+\ell}\|a_k^j\|_{ H^{m+1+N-k,p}(\s^{d-1})}+\|f\|_{W^{m,p}_{\gamma_N}}\,.
\]
As before, we generally write $\A^{m,p}_{n,N;\ell}$ when the domain $\R^d$ is understood.

Now consider vector fields $\A^{m,p}_{n,N;\ell}(\R^d,\R^d)$ whose components are functions in the asymptotic space 
$\A^{m,p}_{n,N;\ell}$;  when the domain $\R^d$ and the fact that they are vector fields are understood, we again 
simply write $\A^{m,p}_{n,N;\ell}$.
Theorem \ref{th:main} implies that the asymptotic space of vector fields $\A^{m,p}_{N;0}\equiv\A^{m,p}_{0,N;0}$ is 
preserved by the Euler flow. In order to prove this fact and the stated well-posedness, we will consider the following
{\em groups of asymptotic diffeomorphisms} 
(or simply, {\em asymptotic groups}).

\medskip\medskip

\noindent{\em Groups of asymptotic diffeomorphisms.}
Denote by $\mathop{\rm Diff}^1_+(\R^d)$ the group of orientation preserving $C^1$-diffeomorphisms of $\R^d$.
For any integers $0\le n\le N$, $-n\le \ell\le 0$, and $m>1+\frac{d}{p}$, define
\begin{equation}\nonumber
\A D^{m,p}_{n,N;\ell}:=\A D^{m,p}_{n,N;\ell}(\R^d):=\big\{\varphi\in\text{\rm Diff}^1_+(\R^d)\,\big|\,
\varphi(x)=x+w(x), \, w\in\A^{m,p}_{n,N;\ell}\big\}\,.
\end{equation}
For simplicity,  we write $\A D^{m,p}_{N;\ell}$ for $\A D^{m,p}_{0,N;\ell}$.
The differential structure of $\A D^{m,p}_{n,N;\ell}$ is inherited in a natural way from the differential structure of
the Banach space $\A^{m,p}_{n,N;\ell}$. In fact, $\A D^{m,p}_{n,N;\ell}$ is an {\em open subset} in 
$\A^{m,p}_{n,N;\ell}$. To see this, for any $h\in C^1(\R^d,\R^d)$ denote
\[
[{\rm d}_x h]:=\left(\frac{\partial h^k}{\partial x_j}(x)\right)_{1\leq j,k\leq d}
\]
the Jacobian matrix of the map $h : \R^d\to\R^d$.
Then any $\varphi={\rm id}+w\in\A D^{m,p}_{n,N;\ell}$ is identified with $w\in\A^{m,p}_{n,N;\ell}$ such that 
for any $x\in\R^d$, $\det\big(\Id+[{\rm d}_x w]\big)>0$,\footnote{Throughout this paper we use id to represent 
the identity map and $\Id$
to represent the identity matrix.} which is an open condition in $\A^{m,p}_{n,N;\ell}$.
In this way, $\A D^{m,p}_{n,N;\ell}$ becomes a Banach manifold modeled on $\A^{m,p}_{n,N;\ell}$.
Moreover, this allows us to identify in a natural way the tangent bundle $T\big(\A D^{m,p}_{N;\ell}\big)$ of 
$\A D^{m,p}_{N;\ell}$ with the direct product $\A D^{m,p}_{N;\ell}\times\A^{m,p}_{N;\ell}$,
\[
T\big(\A D^{m,p}_{N;\ell}\big)\equiv\A D^{m,p}_{N;\ell}\times\A^{m,p}_{N;\ell}\,.
\]
In a similar way one also defines the classes of asymptotic diffeomorphisms without log-terms 
$\A D^{m,p}_{n,N}$ (see \cite[Section 2]{McOwenTopalov2}). 

\begin{Rem}
The restriction on $\ell$, namely $-n\le \ell\le 0$, that appears in the definition of the asymptotic groups 
ensures that $\A^{m,p}_{n,N;\ell}$ is a Banach algebra with respect to point-wise multiplication
(see Lemma \ref{lem:properties_log_spaces} in Appendix \ref{sec:appendix_properties}).
\end{Rem} 

\medskip

For the convenience of the reader, we collect below the main results on the groups of 
asymptotic diffeomorphisms that are used in the paper. 
The following Theorem is proved in \cite{McOwenTopalov2}.

\begin{Th}\label{th:the_group}
Assume that  $m>1+\frac{d}{p}$, $0\le n\le N$,  and $-n\le\ell\le 0$. Then we have:
\begin{itemize}
\item[$(a)$] The composition map
$(u,\varphi)\mapsto u\circ\varphi$, $\A^{m,p}_{n,N;\ell}\times \A D^{m,p}_{N;\ell}\to\A^{m,p}_{n,N;\ell}$,
and the inverse map,
$\varphi\mapsto\varphi^{-1}$, $\A D^{m+1,p}_{n,N;\ell}\to\A D^{m+1,p}_{n,N;\ell}$,
are continuous;

\item[$(b)$] The maps $(u,\varphi)\mapsto u\circ\varphi$, 
$\A^{m+1,p}_{n,N;\ell}\times \A D^{m,p}_{N;\ell}\to\A^{m,p}_{n,N;\ell}$, and
$\varphi\mapsto\varphi^{-1}$, $\A D^{m+1,p}_{n,N;\ell}\to\A D^{m,p}_{n,N;\ell}$
are $C^1$-smooth maps.
\end{itemize}
\end{Th}

\begin{Coro}\label{coro:topological_group}
For $m>2+\frac{d}{p}$ the Banach manifold  $\A D^{m,p}_{n,N;\ell}$ is a topological group with respect to 
the composition of diffeomorphisms. 
\end{Coro}

\begin{Rem}
Note that similar results also hold for $\A D^{m,p}_{n,N}$ (see \cite{McOwenTopalov2}).
\end{Rem}

\begin{Rem}
Generally, regularity $m>1+d/p$ is sufficient for a collection of diffeomorphisms of Sobolev class $H^m(\R^d,\R^d)$
(without weights) to form a topological group under composition  (see \cite{IKT}, or \cite{EM} for the case 
of a compact manifold). The stronger condition $m>2+d/p$ in Corollary \ref{coro:topological_group} stems from 
the fixed point argument used in \cite{McOwenTopalov2} to prove the existence of inverses when dealing with 
asymptotics. It should be possible to obtain Corollary \ref{coro:topological_group} for $m>1+d/p$, but
\cite{McOwenTopalov2} and the present paper are concerned more with the asymptotic structure of solutions
than minimal regularity requirements. In fact, for technical reasons (see Remarks \ref{rem:m>3+d/p} and \ref{rem2:m>3+d/p}), in the rest of this paper we assume that $\bf m>3+\frac{d}{p}$.
\end{Rem}

\medskip

\noindent{\em Volume preserving diffeomorphisms {\&} Asymptotic ODEs.}
In Section \ref{sec:ode} we prove the following

\begin{Prop}\label{prop:ode}
Assume that $u\in C^0\big([0,T],\A^{m,p}_{n,N;\ell}\big)$ for some $T>0$, $0\le n\le N$, and $-n\le\ell\le 0$. 
Then there exists a unique solution $\varphi\in C^1\big([0,T],\A D^{m,p}_{n,N;\ell}\big)$ of the equation,
\[
\dt\varphi=u\circ\varphi,\quad\varphi|_{t=0}=\id\,.
\]
The same result also holds for the asymptotic groups $\A D^{m,p}_{n,N}$.
\end{Prop}

In addition to the group $\A D^{m,p}_{n,N;\ell}$, we also study the topological subgroup of 
{\em volume preserving asymptotic diffeomorphisms},
\[
\accentset{\,\,\,\circ}{\A} D^{m,p}_{n,N;\ell}:=\big\{\varphi\in\A D^{m,p}_{n,N;\ell}\,\big|\,
\forall x\in\R^d,\,\det[{\rm d}_x\varphi]=1\big\}\,.
\]
Denote by $\accentset{\,\,\,\circ}{\A}^{m,p}_{n,N;\ell}$ the space of {\em divergence free asymptotic vector fields} in 
$\A^{m,p}_{n,N;\ell}$,
\[
\accentset{\,\,\,\circ}{\A}^{m,p}_{n,N;\ell}:=\big\{u\in\A^{m,p}_{n,N;\ell}\,\big|\,\Div u=0\big\}\,.
\]
Clearly, $\accentset{\,\,\,\circ}{\A}^{m,p}_{n,N;\ell}$ is a closed subspace in ${\A}^{m,p}_{n,N;\ell}$.

As a consequence of Proposition \ref{prop:ode} we also obtain in Section \ref{sec:ode} the following

\begin{Coro}\label{coro:ode_volume_preserving}
Assume that $u\in C^0\big([0,T],\accentset{\,\,\,\circ}{\A}^{m,p}_{n,N;\ell}\big)$ for some $T>0$,
$0\le n\le N$, and $-n\le\ell\le 0$. Then there exists a unique solution 
$\varphi\in C^1\big([0,T],\accentset{\,\,\,\circ}{\A} D^{m,p}_{n,N;\ell}\big)$
of the equation,
\[
\dt\varphi=u\circ\varphi,\quad\varphi|_{t=0}=\id\,.
\]
\end{Coro}

\begin{Rem}
We prove in Section \ref{sec:volume_preserving_diffeomorphisms} that $\accentset{\,\,\,\circ}{\A}D^{m,p}_{n,N;0}$
is a real analytic submanifold in ${\A}D^{m,p}_{n,N;0}$ (Theorem \ref{th:volume_preserving_diffeomorphisms}).
However, note that this result is not used in the proof of Theorem \ref{th:main}.
\end{Rem}

\begin{Rem}
The same results also hold for the topological subgroup of volume preserving asymptotic diffeomorphisms
without log-terms $\accentset{\,\,\,\circ}{\A} D^{m,p}_{n,N}$, that is defined in a similar way as above. 
\end{Rem}

\medskip

\noindent Proposition \ref{prop:ode} and Corollary \ref{coro:ode_volume_preserving} allow us to define the 
{\em Lie group exponential map} on $\A D^{m,p}_{n,N;\ell}$.
More notably, these two statements allow us to use Arnold's approach to fluid dynamics \cite{Arn,EM,BB} and to 
re-write the Euler equation \eqref{eq:euler} as a dynamical system on the tangent bundle 
$T\big(\A D^{m,p}_{0,N;0}\big)\equiv \A D^{m,p}_{0,N;0}\times\A^{m,p}_{0,N;0}$ 
of the asymptotic group $\A D^{m,p}_{0,N;0}$. In order to do this we first need to study in more detail
the properties of the Laplace operator in asymptotic spaces.

%%%%%%%%%%%%%%%%%%%%%%%%%%%%%%%%%%%%%%
\section{The Laplace operator in asymptotic spaces}\label{sec:laplace_operator}
In this section we study the (Euclidean) Laplace operator
$\Delta:=\sum\limits_{k=1}^d\frac{\partial^2}{\partial x_k^2}$ acting on asymptotic spaces in $\R^d$. 
At the end of the section we prove several properties of the right translation.

\medskip

It follows from Lemma \ref{lem:properties_log_spaces} in Appendix \ref{sec:appendix_properties} that for any 
$0\le n\le N$, $l\ge -n$ the mapping
$
\Delta : \A^{m,p}_{n,N;\ell}\to\A^{m-2,p}_{n+2,N+2;\ell-2}
$
is bounded. Let us consider the special case, $n=1$, $\ell=-1$, and replace $N$ by $N+1$. Since the functions in
$\A^{m,p}_{1,N+1;-1}$ are bounded on $\R^d$ and vanish at infinity, the operator
\[
\Delta : \A^{m,p}_{1,N+1;-1}\to \A^{m-2,p}_{3,N+3;-3}
\] 
is injective. To obtain a surjective map, we need to enlarge the domain space. Below we will define a space
${\widehat\A}^{m,p}_{1,N+1;-1}\supseteq \A^{m,p}_{1,N+1;-1}$ for which
\begin{equation}\label{eq:laplace_isomorphism}
\Delta : {\widehat\A}^{m,p}_{1,N+1;-1}\to \A^{m-2,p}_{3,N+3;-3}
\end{equation}
is an isomorphism. 
The functions that we add are of the form
\begin{equation}\label{def:ujk}
u_k^{j}:=\chi(r)\,\frac{e_k^{j}(\theta)(\log r)^k}{r^k}\,, \ \hbox{for}\ k=\max\{1,d-2\},...,N+1\ \hbox{and}\ 
j=1,\dots,s_k,
\end{equation}
where $s_k\ge 1$ denotes the multiplicity of the eigenvalue $\lambda_k:=k(k+2-d)\geq 0$ of the Laplace-Beltrami
operator $-\Delta_S$ on the unit sphere $\s^{d-1}$, and $\{e_k^{j}\}_{j=1}^{s_k}$ is a basis for 
the corresponding eigenspace. 

\begin{Rem}
For integer $k\ge d-2$ the quadratic expression $\lambda_k=k(k+2-d)$ is non-negative and enumerates
the eigenvalues of the ``positive'' Laplace-Beltrami operator $-\Delta_S$. These eigenvalues are usually enumerated by 
the quadratic expression $\mu_l=l(l+d-2)$, $l\ge 0$.
\end{Rem}

\noindent In order to compute $\Delta$ on these functions, we use the following formulas: 
for integers $k\ge 0$ and $l\ge 2$,
\begin{equation}\label{eq:Delta(l>= 2)}
\begin{array}{l}
\Delta\Big(\frac{a(\theta) (\log r)^l}{r^k}\Big)
=\frac{l(l-1) a(\theta) (\log r)^{l-2}-l(2k+2-d)a(\theta) (\log r)^{l-1}+
\big((\Delta_S a)(\theta)+k(k+2-d) a(\theta)\big)(\log r)^l}
{r^{k+2}}\,,
\end{array}
\end{equation}
while for $l=1$ and $l=0$ one has respectively,
\begin{equation}\label{eq:Delta(l=1)}
\Delta\Big(\frac{a(\theta)\log r}{r^k}\Big)=
\frac{-(2k+2-d)a(\theta)+\big((\Delta_S \, a)(\theta)+k(k+2-d) a(\theta)\big)\log r}{r^{k+2}}
\end{equation}
and
\begin{equation}\label{eq:Delta(l=0)}
\Delta\Big(\frac{a(\theta)}{r^k}\Big)=\frac{(\Delta_S \,a)(\theta)+k(k+2-d) a(\theta)}{r^{k+2}}\, .
\end{equation}
Now if we replace $a(\theta)$ by an eigenfunction $e_k^{j}(\theta)$ and let $l=k$ as in \eqref{def:ujk}, 
then all the terms involving $\Delta_S$, i.e.\ the highest power of $\log r$, drop out and we find
\begin{equation}\label{Lap(ujk)}
\Delta(u_k^{j})\in\A^{m-2,p}_{3,N+3;-3}\,.
\end{equation}

\noindent For $d\geq 3$, let us define
\begin{subequations}
\begin{equation}\label{eq:A-hat_structure}
{\widehat\A}^{m,p}_{1,N+1;-1}:=\A^{m,p}_{1,N+1;-1}\oplus
\mathop{\rm span}_{\R} \big\{u_k^j\,\big|\, 1\leq j\leq s_k, \ d-2\leq k\leq N+1\big\}
\end{equation}
and observe that
\begin{equation}\label{eqq:A-hat_structure}
{\widehat\A}^{m,p}_{1,N+1;-1}\subseteq \A_{1,N+1;0}^{m,p}.
\end{equation}
\end{subequations}
As a consequence of \eqref{Lap(ujk)}, we have that \eqref{eq:laplace_isomorphism}
is well-defined and bounded. Moreover, \eqref{eq:laplace_isomorphism} is  an injection since $k\geq d-2\geq 1$ implies 
that all the $u_k^j$ vanish at infinity. The proposition below will show that \eqref{eq:laplace_isomorphism} is 
in fact an isomorphism.

\smallskip

When $d=2$, convolution by the fundamental solution even on functions with compact support can generate 
a log-term, so we need to replace $\A^{m,p}_{1,N+1;-1}$ by the asymptotic space,
\begin{equation}\label{eq:A-1*}
\A^{m,p}_{1^*,N+1;-1}:=\mathop{\rm span}\limits_{\R}\big\{\chi(r)\log r\big\}\oplus\A^{m,p}_{1,N+1;-1}\,.
\end{equation}
Now we define for $d=2$,
\begin{subequations}
\begin{equation}\label{eq:A-hat_structure(d=2)}
{\widehat\A}^{m,p}_{1,N+1;-1}:=\A^{m,p}_{1^*,N+1;-1}\oplus
\mathop{\rm span}_{\R}\big\{u_k^j\,\big|\, 1\leq j\leq s_k, \ 1\leq k\leq N+1\big\},
\end{equation}
and observe that
\begin{equation}\label{eqq:A-hat_structure(d=2)}
{\widehat\A}^{m,p}_{1,N+1;-1}\subseteq\A^{m,p}_{1^*,N+1;0},\quad
\A^{m,p}_{1^*,N+1;0}:=\mathop{\rm span}\limits_{\R}\big\{\chi(r)\log r\big\}
\oplus\A^{m,p}_{1,N+1;0}.
\end{equation}
\end{subequations}
Since $\Delta:\A^{m,p}_{1^*,N+1;-1}\to  \A^{m-2,p}_{3,N+3;-3}$, we can again use \eqref{Lap(ujk)} to conclude 
that \eqref{eq:laplace_isomorphism} is well-defined and bounded. 
For $d=2$, functions in ${\widehat\A}^{m,p}_{1,N+1;-1}$ need not vanish at infinity; in fact, by \eqref{eq:A-1*} they
can grow logarithmically at infinity. However, the only harmonic functions on $\R^d$ with at most logarithmic growth are
constants, and ${\widehat\A}^{m,p}_{1,N+1;-1}$ has been purposely defined to not contain the constants. Hence 
\eqref{eq:laplace_isomorphism} is also injective for $d=2$.

\smallskip

The following confirms that \eqref{eq:laplace_isomorphism} is an isomorphism.

\begin{Prop}\label{prop:laplace_operator} 
For any $d\ge 2$, the map $\Delta : {\widehat\A}^{m,p}_{1,N+1;-1}\to \A^{m-2,p}_{3,N+3;-3}$ is an isomorphism.
If $d=2$, then the leading term in the asymptotics of $\Delta^{-1}(u)$, for $u\in\A^{m-2,p}_{3,N+3;-3}$, is
\begin{equation}
\frac{1}{2\pi}\,M(u)\,\chi(r)\log r,
\end{equation}
where $M(u):=\int_{\R^2}u(x)\,{\rm d}x$. 
\end{Prop}

\begin{Rem}
As $e_k^{j}(\theta)$, $k=d-2,\dots,N+1$, $j=1,\dots, s_k$, are eigenfunctions of the Laplace-Beltrami operator on the
unit sphere, we see that they are polynomial in $\theta\equiv\frac{x}{|x|}$. Moreover, for $k=d-2$ we have $s_{d-2}=1$
and $e_{d-2}^1=\const$ (cf.\ \cite{Shubin}).  Note also that, for any $d\ge 2$, the space ${\widehat\A}^{m,p}_{1,N+1;-1}$
does not have asymptotically constant terms of the form $\chi(r)\cdot \const$, with $\const\ne 0$.
\end{Rem}

\noindent{\em Proof of Proposition \ref{prop:laplace_operator}.}
By the open mapping theorem, we need only show that the map
$\Delta : {\widehat\A}^{m,p}_{1,N+1;-1}\to \A^{m-2,p}_{3,N+3;-3}$ is surjective,
i.e.\   for any  $u\in\A^{m-2,p}_{3,N+3;-3}$ we have $u\in{\cal I}:=\Delta({\widehat\A}^{m,p}_{1,N+1;-1})$. 
Without loss of generality we can assume that
\[
u=\chi(r)\,\frac{a_{k+2}^l(\theta) (\log r)^l}{r^{k+2}}+f,\quad f\in W^{m-2,p}_{\gamma_{N+3}},
\quad a_{k+2}^l\in H^{m-1+(N+3)-(k+2)}(\s^{d-1}),
\]
where $1\le k\le N+1$ and $0\le l\le k-1$. For fixed $1\le k\le N+1$, we will use induction in $l$
to show $u\in{\cal I}$.
First consider $l=0$, i.e.\ 
\[
u=\chi(r)\,\frac{a_{k+2}^0(\theta) }{r^{k+2}}+f,\quad f\in W^{m-2,p}_{\gamma_{N+3}},
\quad a_{k+2}^l\in H^{m-1+(N+3)-(k+2)}(\s^{d-1}).
\] 
If $a_{k+2}^0$ is not an eigenfunction of 
$-\Delta_S$ with eigenvalue  $\lambda_k=k(k+2-d)$ then there exists $b_k^0\in H^{m+1+(N+1)-k}(\s^{d-1})$ 
such that $(\Delta_S+\lambda_k)\,b_k^0=a_{k+2}^0$. In this case, formula \eqref{eq:Delta(l=0)} implies that
\[
u=\chi\Delta\big(b_k^0/r^k\big)+f
=\Delta\big(\chi b_k^0/r^k\big)+{\tilde f},\quad\hbox{where}\  {\tilde f}\in W^{m-2,p}_{\gamma_{N+3}}\,.
\]
By Lemma \ref{lem:inverse_laplacian_remainder_space} in Appendix \ref{sec:appendix_properties} we know 
$\tilde f=\Delta K \tilde f=\Delta v$, 
where $v\in \widehat\A^{m,p}_{1,N+1,-1}$ (in fact, $v\in \A^{m,p}_{d-2,N+1}$ if $d\geq 3$, and 
$v\in \A^{m,p}_{1^*,N+1}$ if $d=2$).
This means that $u\in{\cal I}$. On the other hand,
if $a_{k+2}^0$ is an eigenfunction of $-\Delta_S$ with eigenvalue $\lambda_k$,\footnote{As the eigenvalues of
$-\Delta_S$ are non-negative, this could happen only if $k\ge d-2$.} we see from \eqref{eq:Delta(l=1)} above that
\[
u=\frac{\chi}{d-2-2k}\,\Delta\left(\frac{a_{k+2}^0\log r}{r^{k}}\right)+f
=\Delta\left(\frac{\chi \,a_{k+2}^0\log r}{(d-2-2k)\,r^k}\right)+{\tilde f},\ \hbox{where}\ {\tilde f}\in
W^{m-2,p}_{\gamma_{N+3}}\,.
\]
(Note that  $d-2-2k\not=0$ since $k\ge\max\{1,d-2\}$.) 
But we claim that  $\chi \, a^0_{k+2}\log r/r^k \in{\widehat\A}^{m,p}_{1,N+1;-1}$: if $k=1$ this is true since
$\chi \, a_{3}^0\log r/r$ is a linear combination of $\{u^j_1\}_{j=1}^{s_1}$, while for $1<k\leq N+1$ it is true since
$a^0_{k+2}\in C^\infty(\s^{d-1})$ implies $\chi \, a^0_{k+2}\log r/r^k \in{\A}^{m,p}_{1,N+1;-1}$. 
Using Lemma  \ref{lem:inverse_laplacian_remainder_space} as before, we again have $u\in{\cal I}$. This completes
the first step of the induction.

\smallskip

Now fix $0<l_0\le k-1$, and assume that $u\in{\cal I}$ for any $0\le l<l_0$.  Consider
\[
u=\chi(r)\,\frac{a_{k+2}^{l_0}(\theta) (\log r)^{l_0}}{r^{k+2}}+f,\quad f\in W^{m-2,p}_{\gamma_{N+3}},
\quad a_{k+2}^{l_0}\in H^{m-1+(N+3)-(k+2)}(\s^{d-1})\,.
\]
If $a_{k+2}^{l_0}$ is not an eigenfunction of $-\Delta_S$ with eigenvalue $\lambda_k$, we can find
$b_k^{l_0}\in H^{m+1+(N+1)-k}(\s^{d-1})$ satisfying $(\Delta_S+\lambda_k)\,b_k^{l_0}=a_{k+2}^{l_0}$. 
Assuming $l_0\geq 2$, we use \eqref{eq:Delta(l>= 2)} to write
\[
u=\Delta\left(\frac{\chi \, b_k^{l_0}(\log r)^{l_0}}{r^k}\right)+{\tilde f}
+\chi\,l_0\,b_k^{l_0}\left(\frac{(2k+2-d)(\log r)^{l_0-1}-(l_0-1)(\log r)^{l_0-2}}{r^{k+2}}\right),
\]
where ${\tilde f}\in W^{m-2,p}_{\gamma_{N+3}}$. (If $l_0=1$, we use \eqref{eq:Delta(l=1)} to get a similar, 
but simpler, expression.)
By the induction hypothesis, the third term on the right is in ${\cal I}$, and we can
again use Lemma  \ref{lem:inverse_laplacian_remainder_space} to conclude that the second term on the right 
is in ${\cal I}$; hence  $u\in{\cal I}$.
Otherwise, when $a_{k+2}^{l_0}$ is an eigenfunction of $-\Delta_S$ with eigenvalue $\lambda_k$, we
obtain from formula \eqref{eq:Delta(l>= 2)}  that
\[
u= \Delta\left(\frac{\chi\, a_{k+2}^{l_0}(\log r)^{l_0+1}}{(l_0+1)(d-2-2k)\,r^k}\right)+{\tilde f}
+\frac{\chi\, l_0 \,a^{l_0}_{k+2}(\log r)^{l_0-1}}{(2k+2-d)\,r^{k+2}},
\]
where ${\tilde f}\in W^{m-2,p}_{\gamma_{N+3}}$. (Again recall $2k+2-d\not=0$.) By the induction hypothesis,
the third term on the right is in $\cal I$. We also claim that 
$\chi \,a_{k+2}^{l_0}(\log r)^{l_0+1}/r^k\in{\widehat\A}^{m,p}_{1,N+1;-1}$: if $l_0+1=k$, this is true since
$\chi \,a_{k+2}^{l_0}(\log r)^{k}/r^k$ is a linear combination of $\{u_k^j\}_{j=1}^{s_k}$, while for $l_0+1<k$ it 
is true since $a_{k+2}^{l_0}\in C^\infty(S^{d-1})$ implies 
$\chi \,a_{k+2}^{l_0}(\log r)^{l_0+1}/r^k\in{\A}^{m,p}_{1,N+1;-1}$.
Using Lemma \ref{lem:inverse_laplacian_remainder_space} again, we conclude that $u\in{\cal I}$.
This completes the induction, and hence the proof of surjectivity.

Let us now prove the second statement of the Proposition. By Green's identity, for any $3<R<\infty$
\begin{equation}\label{eq:mean_2pi}
\iint_{\R^2}\Delta\big(\chi(r)\log r\big)\,{\rm d}x=\iint_{\{|x|\le R\}}\Delta\big(\chi(r)\log r\big)\,{\rm d}x=
\int_{\{|x|=R\}}\frac{\partial}{\partial r}\big(\log r\big)\,{\rm d}s=2\pi
\end{equation}
where ${\rm d}s$ is the length form on the circle $\big\{x\in\R^2\,\big|\,|x|=R\big\}$.
Similarly, for any $f\in\A^{m,p}_{1,N+1;0}$ we have 
\begin{equation}\label{eq:mean_zero}
\iint_{\R^2}\Delta f\,{\rm d}x=\lim\limits_{R\to\infty}\int_{\{|x|=R\}}(\partial f/\partial r)\,{\rm d}s=0
\end{equation}
as $\partial f/\partial r=O(1/r^2)$.
Now, take an arbitrary $u\in\A^{m-2,p}_{3,N+3;-3}$. We already proved that 
$\Delta^{-1}(u)\in{\widehat\A}^{m,p}_{1,N+1;-1}$. Hence, 
\[
\Delta^{-1}(u)=c\,\chi(r)\log r+f
\]
where $f\in\A^{m,p}_{1,N+1;0}$. This gives, $u=c\,\Delta\big(\chi(r\log r\big)+\Delta f$.
By integrating this equality over $\R^2$ we see, in view of \eqref{eq:mean_2pi} and \eqref{eq:mean_zero}, 
that $M(u)=2\pi c$. This completes the proof of Proposition \ref{prop:laplace_operator}.
\finishproof

\begin{Rem}\label{rem:laplace_operator}
Define
\begin{equation*}
{\widehat\A}^{m,p}_{0,N;0}:=\A^{m,p}_{{\widetilde 0},N;0}\oplus
\mathop{\rm span}_{\R}\big\{v_k^j\,\big|\,1\le j\le s_k, d-2\le k\le N\big\}
\end{equation*}
where $v_0^1:=\chi(r) (\log r)^2$ and
$v_k^j:=\chi(r) e^j_k(\theta) (\log r)^{k+1}/r^k$ for $1\le j\le s_k$, $1\le k\le N$, and, for $d\ge 3$,
\begin{equation*}
\A^{m,p}_{{\widetilde 0},N;0}:=\A^{m,p}_{[0],N;0},\quad\text{where}\quad
\A^{m,p}_{[0],N;0}:=\Big\{u\in\A^{m,p}_{0,N;0}\,\Big|\,\int_{\omega\in\s^{d-1}} a_0(\omega)\,{\rm d}s(\omega)=0\Big\},
\end{equation*}
and, for $d=2$,
\begin{equation*}
\A^{m,p}_{{\widetilde 0},N;0}:=\mathop{\rm span}_{\R}\big\{\chi(r) \log r\big\}\oplus\A^{m,p}_{[0],N;0}\,.
\end{equation*}
Clearly, ${\widehat\A}^{m,p}_{0,N;0}\subseteq\A^{m,p}_{0,N;2}$.
The arguments used in the proof of Proposition \ref{prop:laplace_operator} imply that 
\begin{equation}\label{eq:laplace_isomorphism_general}
\Delta : {\widehat\A}^{m,p}_{0,N;0}\to \A^{m-2,p}_{2,N+2;-2},
\end{equation}
is an isomorphism. Moreover we have the following commutative diagram,
\[
\begin{array}{ccccc}
\A^{m+1,p}_{0,N;2}&\supseteq&{\widehat\A}^{m+1,p}_{0,N;0}&\stackrel{\Delta}{\rightarrow}&
\A^{m-1,p}_{2,N+2;-2}\\
\partial_j\downarrow&&\partial_j\downarrow&&\partial_j\downarrow\\
\A^{m,p}_{1,N+1;1}&\supseteq&{\widehat\A}^{m,p}_{1,N+1;-1}&\stackrel{\Delta}{\rightarrow}&
\A^{m-2,p}_{3,N+3;-3}\\
\end{array}
\]
where the downward arrow denotes the partial differentiation with respect to the $j$-th variable in $\R^d$. 
In particular,
we see that $\Delta^{-1}$ commutes with $\frac{\partial}{\partial x_j}$, where $\Delta^{-1}$ denotes the inverse
of the maps $\Delta$ appearing in the diagram above.
\end{Rem}
\noindent In what follows, we will denote by $\Delta^{-1}$ the inverse of
\eqref{eq:laplace_isomorphism} or \eqref{eq:laplace_isomorphism_general}, depending on the context.

\medskip\medskip

For any $\varphi\in\A D^{m,p}_{0,N;0}$ consider the {\em conjugate Laplace} operator,
\begin{equation}\label{eq:conjugate_laplacian}
\Delta_\varphi:=R_\varphi\circ\Delta\circ R_{\varphi^{-1}}
\end{equation}
where $R_\varphi v$ denotes the {\em right-translation} of a function $v$ by 
$\varphi$, $R_\varphi v:=v\circ\varphi$.
We begin by studying properties of $R_\varphi$ acting on asymptotic spaces.

\begin{Lem}\label{lem:composition*}
Let $a(\theta)=a(\theta_1,...,\theta_d)$ be a polynomial function in the variables $(\theta_1,...,\theta_d)$.
Take $\theta\equiv\frac{x}{|x|}$ and consider the homogeneous of degree zero function $a(\theta)$ in $x$. 
Then for any $\varphi\in\A D^{m,p}_{0,N;0}$, $k\ge 0$ and $\ell\ge -k$, and $N^-<N$,
\[
\Big(\chi(r)\,\frac{a(\theta) (\log r)^{k+\ell}}{r^k}\Big)\circ\varphi-\chi(r)\,\frac{a(\theta) (\log r)^{k+\ell}}{r^k}\in
\A^{m,p}_{k+1,N^-+k+1;\ell-1}\,.
\]
In the case when $k+\ell=0$, the statement above is true also with $N^-=N$.
Moreover, 
\[
\big(\chi(r)\log r\big)\circ\varphi-\chi(r)\log r\in\A^{m,p}_{1,N+1;-1}\,.
\]
\end{Lem}

\begin{proof}
Take $\varphi\in\A D^{m,p}_{0,N;0}$. Then $\varphi(x)=x+w(x)$ where $w\in\A^{m,p}_{0,N;0}$.
Let us first consider the composition $r\circ\varphi$.
We have,
\begin{eqnarray}
(r\circ\varphi)(x)&=&|x+w(x)|=r\,\Big|\theta+\frac{w(x)}{r}\Big|\nonumber\\
&=&r\,\sqrt{1+2\big(\theta,w/r\big)+\big(w/r,w/r\big)}\,,\label{eq:composition_r1}
\end{eqnarray}
where $(\cdot,\cdot)$ denotes the Euclidean scalar product in $\R^d$.
In view of Lemma \ref{lem:asymptotic_spaces_complement} $(b)$ in Appendix \ref{sec:appendix_properties}, 
\[
w/r\in\A^{m,p}_{1,N+1;-1}(B_R^c),\quad R>0,
\]
where $B_R^c:=\big\{x\in\R^d\,\big|\,|x|>R\big\}$ and the space $\A^{m,p}_{1,N+1;-1}(B_R^c)$ equipped with the 
norm \eqref{eq:the_norm_complement} is defined in Appendix \ref{sec:appendix_properties}.
As $w/r=o(1)$ as $|x|\to\infty$ we can choose $r_0>0$ such that $|w(x)/r|<1/4$ for any $|x|>r_0$. Then
for any $|x|>r_0$ we have $\big|2\big(\theta,w/r\big)+\big(w/r,w/r\big)\big|< 3/4$, and hence in view of 
\eqref{eq:composition_r1}, 
\begin{equation}\label{eq:the_series1}
r\circ\varphi=r\,\Big(1+\sum_{k=1}^\infty c_k\Big(2\big(\theta,w/r\big)+\big(w/r,w/r\big)\Big)^k\Big),
\end{equation}
where the series converges uniformly in $|x|>r_0$, and, by Cauchy's estimate, $|c_k|\le\const\,(4/3)^k$ for some
constant independent of $k\ge 1$. 
Denote
\[
{\widetilde w}:=2\big(\theta,w/r\big)+\big(w/r,w/r\big).
\]
It follows from Lemma \ref{lem:asymptotic_spaces_complement} $(b)$
and Lemma \ref{lem:properties_log_spaces} that
\[
{\widetilde w}\in\A^{m,p}_{1,N+1;-1}(B_R^c)\,.
\]
(As $\theta\in\A^{m,p}_{0,N+1;0}(B_R^c)$ and $w/r\in\A^{m,p}_{1,N+1;-1}(B_R^c)$ we have from 
Lemma \ref{lem:asymptotic_spaces_complement} $(b)$ and Lemma \ref{lem:properties_log_spaces} $(d')$ that 
$(\theta,w/r\big)\in\A^{m,p}_{1,N+1;-1}(B_R^c)$.)
Moreover, in view of \eqref{eq:the_series1} we have
\begin{equation}\label{eq:the_series2}
r\circ\varphi=r\Big(1+\sum_{k=1}^{N+1} c_k{\widetilde w}^k+
\sum_{k\ge N+2} c_k{\widetilde w}^{N+2}{\widetilde w}^{k-N-2}\Big)\,.
\end{equation}
By Lemma \ref{lem:asymptotic_spaces_complement} $(b)$ and Lemma \ref{lem:properties_log_spaces} we have
$\sum_{k=1}^{N+1} c_k{\widetilde w}^k\in\A^{m,p}_{1,N+1;-1}(B_R^c)$.
Further, note that by Lemma \ref{lem:asymptotic_spaces_complement} $(b)$ and 
Lemma \ref{lem:properties_log_spaces} $(a)$ we have that
${\widetilde w}^k\in\A^{m,p}_{k,N+k;-k}(B_R^c)\subseteq W^{m,p}_{\gamma_{N+1}}(B_R^c)$ for any $k\ge N+2$.
This, together with Lemma \ref{lem:asymptotic_spaces_complement} $(c)$ then implies that we can choose 
$r_0>0$ larger if necessary so that, for any $R>r_0>0$ and for any $k\ge N+2$, 
\begin{eqnarray}
\big\|{\widetilde w}^{N+2}{\widetilde w}^{k-N-2}\big\|_{\A^{m,p}_{1,N+1;-1}(B^c_R)}&\le&
\|{\widetilde w}^{N+2}\|_{\A^{m,p}_{1,N+1;-1}(B^c_R)}
\big(C\|{\widetilde w}\|_{\A^{m,p}_{1,N+1;-1}(B^c_R)}\big)^{k-N-2}\nonumber\\
&<&\|{\widetilde w}^{N+2}\|_{\A^{m,p}_{1,N+1;-1}(B^c_R)}\big(3/4\big)^{k-N-2},\label{eq:convergence}
\end{eqnarray}
where $C>0$ is the constant in Lemma \ref{lem:asymptotic_spaces_complement} $(c)$ and where
we used that $\|{\widetilde w}\|_{\A^{m,p}_{1,N+1;-1}(B_R^c)}\to 0$ as $R\to\infty$.
Inequality \eqref{eq:convergence} implies that the series in \eqref{eq:the_series2} converges
in $\A^{m,p}_{1,N+1;-1}(B_R^c)$.
Hence, for any $R>r_0>0$,
\begin{equation}\label{eq:composition_r}
r\circ\varphi=r\,\big(1+{\tilde r}\big),\,\,\,\,{\tilde r}\in\A^{m,p}_{1,N+1;-1}(B^c_R)\,.
\end{equation}
Using  \eqref{eq:composition_r} and arguing in a similar way as above, we see that, for $r_0>0$, 
taken larger if necessary, and for any $R>r_0>0$,
\begin{equation}\label{eq:log_r}
(\log r)\circ\varphi=\log r+{\tilde p},\,\,\,\,{\tilde p}\in\A^{m,p}_{1,N+1;-1}(B^c_R),
\end{equation}
\begin{equation}\label{eq:composition_theta}
\theta\circ\varphi-\theta\in\A^{m,p}_{1,N+1;-1}(B^c_R).
\end{equation}
As $a(\theta)$ is assumed a polynomial function of $\theta$ we obtain from \eqref{eq:composition_theta} that
\begin{equation}\label{eq:composition_a}
a(\theta)\circ\varphi-a(\theta)\in\A^{m,p}_{1,N+1;-1}(B^c_R)\,.
\end{equation}
It follows from \eqref{eq:composition_r}--\eqref{eq:composition_a} and the Banach algebra property of  
$\A^{m,p}_{1,N+1;-1}(B^c_R)$ (see Lemma \ref{lem:asymptotic_spaces_complement} $(b)$) that, for some 
$R>\max\{r_0, 2\}$,
\begin{equation}\nonumber
\Big(\frac{a(\theta) (\log r)^{k+\ell}}{r^k}\Big)\circ\varphi-\frac{a(\theta) (\log r)^{k+\ell}}{r^k}\in
\A^{m,p}_{k+1,N^-+k+1;\ell-1}(B^c_R)
\end{equation}
for any $N^-<N$, and for $N^-=N$ if $k+\ell=0$.
Finally, take $R'>R$, $\varphi\in\A D^{m,p}_{0,N;0}$, and $\zeta\in C^\infty(\R^d)$ with compact support so that 
$\zeta|_{\varphi(B_{R'})}\equiv 1$, where $B_{R'}$ is the open ball of radius $R'$ in $\R^d$ 
centered at the origin. As $\zeta\chi\,\frac{a(\theta) (\log r)^{k+\ell}}{r^k}$ is $C^\infty$-smooth with compact
support we have from Theorem \ref{th:the_group} that 
$\Big(\zeta\chi\,\frac{a(\theta) (\log r)^{k+\ell}}{r^k}\Big)\circ\varphi\in\A^{m,p}_{0,N;0}$. 
This, together with the choice of $\zeta$, then implies
\[
\Big(\chi\,\frac{a(\theta) (\log r)^{k+\ell}}{r^k}\Big)\circ\varphi\Big|_{B_{R'}}=
\Big(\zeta\,\chi\,\frac{a(\theta) (\log r)^{k+\ell}}{r^k}\Big)\circ\varphi\Big|_{B_{R'}}\in H^{m,p}(B_{R'}).
\]
The conclusion in the Lemma then follows from Lemma \ref{lem:asymptotic_spaces_complement} $(a)$.
The last statement of the Lemma follows in the same way from \eqref{eq:log_r}.
\end{proof}

\medskip

\noindent As a consequence of this Lemma we get

\begin{Coro}\label{coro:invariance_1}
Assume that $1+\frac{d}{p}<m_0\le m$ and $n\ge 0$. Then for any $\varphi\in\A D^{m,p}_{0,N;0}$ we have
$R_{\varphi}(\A^{m_0,p}_{n,N+n;-n})\subseteq\A^{m_0,p}_{n,N+n;-n}$ and the map
\[
R_\varphi : \A^{m_0,p}_{n,N+n;-n}\to \A^{m_0,p}_{n,N+n;-n}
\]
is an isomorphism.
\end{Coro}
\begin{proof} Take $\varphi\in\A D^{m,p}_{0,N;0}\subseteq\A D^{m_0,p}_{0,N;0}$.
 In view of Lemma \ref{lem:composition*} there exists $R>0$ such that
\[
\frac{1}{(r\circ\varphi)^n}=\frac{1}{r^n}\big(1+{\tilde r}\big),\,\,\,{\tilde r}\in\A^{m,p}_{0,N;0}(B^c_R)\,.
\]
Take $v\in\A^{m_0,p}_{n,N+n;-n}$. Then for $|x|>R$ we can write $v={\tilde v}/r^n$ where 
${\tilde v}\in\A^{m_0,p}_{0,N;0}(B^c_R)$.
By taking $R>0$ larger if necessary, we see that
\begin{equation}\label{eq:shift}
v\circ\varphi=\frac{1}{(r\circ\varphi)^n}\,{\tilde v}\circ\varphi=
\big({\tilde v}\circ\varphi+{\tilde r}\cdot{\tilde v}\circ\varphi\big)/r^n,
\end{equation}
where  ${\tilde r}\in\A^{m_0,p}_{0,N;0}(B^c_R)$ and ${\tilde v}\circ\varphi\in\A^{m_0,p}_{0,N;0}(B^c_R)$
in view of Theorem \ref{th:the_group} and Lemma \ref{lem:asymptotic_spaces_complement} $(a)$. 
This and Lemma \ref{lem:asymptotic_spaces_complement} $(b)$ then 
imply that  $v\circ\varphi\in\A^{m_0,p}_{n,N+n;-n}(B^c_R)$. Arguing as in the last paragraph of the proof of 
Lemma \ref{lem:composition*}, we conclude from Lemma \ref{lem:asymptotic_spaces_complement} $(a)$ that
$v\circ\varphi\in\A^{m_0,p}_{n,N+n;-n}$.
The boundedness of the map $R_\varphi : \A^{m_0,p}_{n,N+n;-n}\to \A^{m_0,p}_{n,N+n;-n}$ follows from 
\eqref{eq:shift}, Lemma \ref{lem:composition*}, Lemma \ref{lem:asymptotic_spaces_complement} $(a)$ and $(b)$,
and Theorem \ref{th:the_group}. Clearly, this map also is invertible with bounded inverse 
$R_{\varphi^{-1}} : \A^{m_0,p}_{n,N+n;-n}\to \A^{m_0,p}_{n,N+n;-n}$.
\end{proof}

\begin{Coro}\label{coro:invariance_2}
For any $\varphi\in\A D^{m,p}_{0,N;0}$, we have
$R_{\varphi}({\widehat\A}^{m,p}_{1,N+1;-1})\subseteq{\widehat\A}^{m,p}_{1,N+1;-1}$, and the map
\[
R_\varphi : {\widehat\A}^{m,p}_{1,N+1;-1}\to {\widehat\A}^{m,p}_{1,N+1;-1}
\]
is an isomorphism.
\end{Coro}
\begin{proof}
Take $\varphi\in\A D^{m,p}_{0,N;0}$ and $v\in{\widehat\A}^{m,p}_{1,N+1;-1}$. Then by 
\eqref{eq:A-hat_structure}, \eqref{eq:A-1*}, and \eqref{eq:A-hat_structure(d=2)}, $v$ can be uniquely decomposed,
\[
v=v_1+v_2,
\]
where $v_1\in\A^{m,p}_{1,N+1;-1}$ and $v_2$ is a linear combination with constant coefficients of 
the functions
\[
u_k^j=\chi(r)\,\frac{e_k^j(\theta)(\log r)^k}{r^k},\quad\max\{1,d-2\}\le k\le N+1,\quad 1\le j\le s_k,
\]
and the function $u_0^1=\chi(r)\log r$ when $d=2$.
As the coefficients $e_k^j(\theta)$ are polynomial functions of $\theta\equiv\frac{x}{|x|}$, we can apply
Lemma \ref{lem:composition*} to conclude that for $\max\{1,d-2\}\le k\le N+1$,
\begin{equation}\label{eq:v_2'}
u_k^j\circ\varphi=u_k^j+{\tilde u}_k^j,\quad {\tilde u}_k^j\in\A^{m,p}_{1,N+1;-1},
\end{equation}
and, for $d=2$,
\begin{equation}\label{eq:v_2''}
u_0^1\circ\varphi=u_0^1+{\tilde u}_0^1,\quad {\tilde u}_0^1\in\A^{m,p}_{1,N+1;-1}\,.
\end{equation}
In view of Corollary \ref{coro:invariance_1}, we have $v_1\circ\varphi\in\A^{m,p}_{1,N+1;-1}$
and the map $v_1\mapsto v_1\circ\varphi$,  $\A^{m,p}_{1,N+1;-1}\to\A^{m,p}_{1,N+1;-1}$,
is bounded. The projection $v\mapsto v_1$ is also bounded. This together with \eqref{eq:v_2'} and 
\eqref{eq:v_2''} then implies the statement of the Corollary.
\end{proof}

\noindent Corollary \ref{coro:invariance_1} and Corollary \ref{coro:invariance_2} then give

\begin{Coro}\label{coro:conjugation_isomorphism}
For any  given $\varphi\in\A D^{m,p}_{0,N;0}$,
\[
\Delta_\varphi : {\widehat\A}^{m,p}_{1,N+1;-1}\to\A^{m-2,p}_{3,N+3;-3}
\]
is an isomorphism with inverse $(\Delta_\varphi)^{-1}=R_\varphi\circ\Delta^{-1}\circ R_{\varphi^{-1}}$,
where $\Delta^{-1} : \A^{m-2,p}_{3,N+3;-3}\to {\widehat\A}^{m,p}_{1,N+1;-1}$ is the inverse of the map 
\eqref{eq:laplace_isomorphism}.
\end{Coro}

\begin{Rem}\label{rem:m>3+d/p}
Here we use the condition $m>3+\frac{d}{p}$ when applying Corollary \ref{coro:invariance_1}.
\end{Rem}

%%%%%%%%%%%%%%%%%%%%%%%%%%%%%%%%%%
\section{The Euler vector field}\label{sec:euler_vector_field}
In this section we associate to the Euler equation \eqref{eq:euler} a dynamical system on the tangent bundle
$T\big(\A D^{m,p}_{0,N;0}\big)\equiv \A D^{m,p}_{0,N;0}\times\A^{m,p}_{0,N;0}$ 
of the asymptotic group $\A D^{m,p}_{0,N;0}$. 
The regularity of this dynamical system and its relation to the Euler equation is discussed in 
the following sections.

\medskip

Let us first discuss the notion of solution. 
Take $u_0\in\accentset{\,\,\,\circ}\A^{m,p}_{0,N;0}\subseteq\A^{m,p}_{0,N;0}$ and recall that
the subspace $\accentset{\,\,\,\circ}\A^{m,p}_{0,N;0}$ is closed in $\A^{m,p}_{0,N;0}$. 
We say that
$u\in C^0\big([0,T],\accentset{\,\,\,\circ}\A^{m,p}_{0,N;0}\big)\cap 
C^1\big([0,T],\accentset{\,\,\,\circ}\A^{m-1,p}_{0,N;0}\big)$ 
is a {\em solution} of the Euler equation if $u$ satisfies \eqref{eq:euler} for some 
${\rm p}\in C^0\big([0,T],S'(\R^d)\big)$, where $S'(\R^d)$ is the space of tempered distributions in $\R^d$. 
In this way, the pressure term ${\rm p}\in C^0\big([0,T],S'(\R^d)\big)$ is not considered a part of the solution.

\smallskip

Assume that  $u\in C^0\big([0,T],\accentset{\,\,\,\circ}\A^{m,p}_{0,N;0}\big)\cap 
C^1\big([0,T],\accentset{\,\,\,\circ}\A^{m-1,p}_{0,N;0}\big)$ is 
a solution of the Euler equation \eqref{eq:euler}. 
Applying the divergence $\Div$ to both sides of equation \eqref{eq:euler}, 
we obtain
\begin{equation}\label{eq:poisson}
\Div\big(u\cdot\nabla u\big)=-\Delta{\rm p}\,.
\end{equation}
In view of the Sobolev embedding $\A^{m,p}_{0,N;0}\subseteq C^2$ and the product rule we have,
\begin{eqnarray}
\Div\big(u\cdot\nabla u\big)&=&\tr\big([{\rm d}u]^2\big)+u\cdot\nabla\big(\Div u\big)\label{eq:the_trick1}\\
&=&\tr\big([{\rm d}u]^2\big),\label{eq:the_trick2}
\end{eqnarray}
where $\tr$ is the trace and $[{\rm d}u]^2$ is the square of the Jacobian matrix $[{\rm d}u]$. 
(Relation \eqref{eq:the_trick2} is standard and appears e.g.\ in \cite{BB}; although quite simple,
it is crucial for our approach.) Combining \eqref{eq:poisson} with 
\eqref{eq:the_trick2} we see that $-\Delta\big(\nabla{\rm p}\big)=\nabla\circ Q(u)$ where 
\begin{equation}\label{eq:Q}
Q(u):=\tr\big([{\rm d}u]^2\big)\,. 
\end{equation}
In view of Lemma \ref{lem:properties_log_spaces} $(c)$ and $(d')$,
\[
Q : \A^{m,p}_{0,N;0}\to\A^{m-1,p}_{2,N+2;-2}
\]
and this map is real-analytic. This, together with Proposition \ref{prop:laplace_operator}, then implies that
\begin{equation}\label{eq:the_pressure_term}
-\nabla{\rm p}=\Delta^{-1}\circ\nabla\circ Q(u),
\end{equation}
where $\Delta^{-1} : \A^{m-2,p}_{3,N+3;-3}\to{\widehat\A}^{m,p}_{1,N+1;-1}$ is a bounded map.
(The space ${\widehat\A}^{m,p}_{1,N+1;-1}$ was defined in Section \ref{sec:laplace_operator} -- see
\eqref{eq:A-hat_structure} and \eqref{eq:A-hat_structure(d=2)}.) 
Plugging \eqref{eq:the_pressure_term} into  \eqref{eq:euler} we obtain 
\begin{equation}\label{eq:euler'}
u_t+u\cdot\nabla u=\Delta^{-1}\circ\nabla\circ Q(u),\quad u|_{t=0}=u_0.
\end{equation}
Using that $\Delta^{-1}$ and $\nabla$ commute (Remark \ref{rem:laplace_operator}) we see that we can set
${\rm p}:=-\Delta^{-1}\circ Q(u)$ where $\Delta^{-1} :  \A^{m-1,p}_{2,N+2;-2}\to{\widehat\A}^{m+1,p}_{0,N;0}$ is
a bounded map and the space ${\widehat\A}^{m+1,p}_{0,N;0}$ is defined in Remark \ref{rem:laplace_operator}. 
In particular, we obtain that 
\begin{equation}\label{eq:pressure}
{\rm p}\in C^0\big([0,T],{\widehat\A}^{m+1,p}_{0,N;0}\big).
\end{equation} 
In fact, a converse statement also holds.

\begin{Lem}\label{lem:no_pressure_term}
The curve $u\in C^0\big([0,T],\A^{m,p}_{0,N;0}\big)\cap 
C^1\big([0,T],\A^{m-1,p}_{0,N;0}\big)$ is a solution of equation \eqref{eq:euler'} with 
initial data $u_0\in\accentset{\,\,\,\circ}\A^{m,p}_{0,N;0}$ if and only if
$u$ belongs to $C^0\big([0,T],\accentset{\,\,\,\circ}\A^{m,p}_{0,N;0}\big)\cap 
C^1\big([0,T],\accentset{\,\,\,\circ}\A^{m-1,p}_{0,N;0}\big)$ 
and is a solution of the Euler equation \eqref{eq:euler}. 
\end{Lem}

\begin{Rem}
In particular, the Lemma implies that any solution 
$u\in C^0\big([0,T],\A^{m,p}_{0,N;0}\big)\cap 
C^1\big([0,T],\A^{m-1,p}_{0,N;0}\big)$
of \eqref{eq:euler'} with divergence free initial data $u_0$
consists of divergence free vector fields.
The advantage of \eqref{eq:euler'} compared with \eqref{eq:euler} is that it does not contain
explicitly an unknown pressure term and that the relation $\Div u=0$ for $t\in [0,T]$ is automatically satisfied.
\end{Rem}
 
\noindent{\em Proof of Lemma \ref{lem:no_pressure_term}.}
In view of the discussion above, we only need to prove the direct implication. Let 
$u\in C^0\big([0,T],\A^{m,p}_{0,N;0}\big)\cap 
C^1\big([0,T],\A^{m-1,p}_{0,N;0}\big)$ be a solution of \eqref{eq:euler'} with initial data 
$u_0\in\accentset{\,\,\,\circ}\A^{m,p}_{0,N;0}$. Using again that
$\Delta^{-1}$ and $\nabla$ commute (Remark \ref{rem:laplace_operator}) we conclude from \eqref{eq:euler'}
that $u$ satisfies 
\[
u_t+u\cdot\nabla u=\nabla\circ\Delta^{-1}\circ Q(u)\,.
\]
Applying the divergence $\Div$ to both sides of this equality we get from \eqref{eq:the_trick1} that
\begin{equation}\label{eq:div-evolution}
(\Div u)_t+u\cdot\nabla(\Div u)=0\,.
\end{equation}
In view of the Sobolev embedding $\A^{m-1,p}_{0,N;0}\subseteq C^1$ we have
$u\in C^1\big([0,T]\times\R^d,\R^d\big)$. Then equation (\ref{eq:div-evolution}) implies that
for any $x\in\R^d$ and for any $t\in[0,T]$,
\[
(\Div u)(t,\varphi(t,x))=(\Div u_0)(x)=0,
\]
where $\varphi(t,\cdot)\in\mathop{\rm Diff}^1_+(\R^d)$ is the flow of $u$ (see Proposition \ref{prop:ode}).
This implies that for any $t\in[0,T]$, $u(t)$ is divergence free.
Hence, 
$u\in C^0\big([0,T],\accentset{\,\,\,\circ}\A^{m,p}_{0,N;0}\big)\cap 
C^1\big([0,T],\accentset{\,\,\,\circ}\A^{m-1,p}_{0,N;0}\big)$
as claimed.
\finishproof

\medskip\smallskip

\noindent In this way we have a bijective correspondence between the solutions of the Euler equations
and the solutions of equation \eqref{eq:euler'} with divergence free initial data. 

\smallskip

Now we are ready to define the {\em Euler vector field}.  
For any 
$(\varphi,v)\in T\big(\A D^{m,p}_{0,N;0}\big)\equiv \A D^{m,p}_{0,N;0}\times \A^{m,p}_{0,N;0}$ define
\begin{equation}\label{eq:E}
{\mathcal E}(\varphi,v):=\big(v,{\mathcal E}_2(\varphi,v)\big)\quad\text{where}\quad
{\mathcal E}_2(\varphi,v):=(R_\varphi\circ\Delta^{-1}\circ\nabla\circ Q\circ R_{\varphi^{-1}})(v)\,.
\end{equation}
It follows from Lemma \ref{lem:properties_log_spaces}, Proposition \ref{prop:laplace_operator}, and
Corollary \ref{coro:invariance_1} and Corollary \ref{coro:invariance_2}, that 
\[
{\mathcal E} : \A D^{m,p}_{0,N;0}\times \A^{m,p}_{0,N;0}\to
\A^{m,p}_{0,N;0}\times{\widehat\A}^{m,p}_{1,N+1;-1}
\]
and moreover
\begin{equation}\label{eq:E-factoring}
{\mathcal E}_2(\varphi,v)=(R_\varphi\circ\Delta^{-1}\circ R_{\varphi^{-1}}\big)\circ
\big(R_\varphi\circ\nabla\circ Q\circ R_{\varphi^{-1}}\big)(v)\,.
\end{equation}

\begin{Rem}\label{rem2:m>3+d/p}
Instead of \eqref{eq:E-factoring}, we could write 
\[
{\mathcal E}_2(\varphi,v)=\big(R_\varphi\circ\nabla\circ\Delta^{-1}\circ R_{\varphi^{-1}}\big)
\circ\big(R_\varphi\circ Q\circ R_{\varphi^{-1}}\big)(v)\,.
\] 
In this form, however, the smoothnes of the map
$(\varphi,v)\mapsto\big(R_\varphi\circ\nabla\circ\Delta^{-1}\circ R_{\varphi^{-1}}\big)(v)$,
$\A D^{m,p}_{0,N;0}\times\A_{1,N+1;-1}^{m-1,p}\to\A_{0,N;0}^{m,p}$,
can not be easily deduced, so we have chosen to use 
\eqref{eq:E}, which requires $m>3+d/p$ (cf.\ Remark \ref{rem:m>3+d/p}).
\end{Rem}

Recall from Section \ref{sec:laplace_operator} that for $d=2$ the elements of ${\widehat\A}^{m,p}_{1,N+1;-1}$ 
can have leading asymptotic terms of the form $c\,\chi(r)\log r$ with $c\ne 0$, and hence 
${\widehat\A}^{m,p}_{1,N+1;-1}\not\subseteq\A^{m,n}_{0,N;0}$. 
However, we have
\begin{Lem}\label{lem:no_log_terms_in_E}
For any $(\varphi,v)\in T\big(\A D^{m,p}_{0,N;0}\big)$ the component ${\mathcal E}_2(\varphi,v)$
in \eqref{eq:E} lies in ${\widehat\A}^{m,p}_{1,N+1;-1}\cap\A^{m,p}_{0,N;0}$. In particular,
it does not have a leading asymptotic term of the form $c\,\chi(r)\log r$ with $c\ne 0$.
\end{Lem}

\proof
We only need to consider the case when $d=2$.
Take an arbitrary $(\varphi,v)\in \A D^{m,p}_{0,N;0}\times\A^{m,p}_{0,N;0}$. 
Then $R_{\varphi^{-1}}(v)\in\A^{m,p}_{0,N;0}$ and, by Lemma \ref{lem:properties_log_spaces},
\[
Q\circ R_{\varphi^{-1}}(v)\in\A^{m-1,p}_{2,N+2;-2}\,\,\,\,\,\,\,\text{and}\,\,\,\,\,\,\,\,\,
\nabla\circ Q\circ R_{\varphi^{-1}}(v)\in\A^{m-2,p}_{3,N+3;-3}.
\]
Denote  $f:=Q\circ R_{\varphi^{-1}}(v)$ and note that
$\nabla f\in\A^{m-2,p}_{3,N+3;-3}\subseteq L^1\cap C^1$.
As $f\in\A^{m-1,p}_{2,N+2;-2}$,
\begin{equation}\label{eq:f-decay}
f(x)=O\big(1/|x|^2\big)\,\,\,\,\text{as}\,\,\,\,\,|x|\to\infty.
\end{equation}
Now, consider the first component $\frac{\partial f}{\partial x_1}$ of $\nabla f$.
As ${\rm d}\big(f(x_1,x_2)\,{\rm d}x_2\big)=\frac{\partial f}{\partial x_1}(x_1,x_2)\,{\rm d}x_1\wedge{\rm d}x_2$
we see from the Stokes' theorem and \eqref{eq:f-decay} that for any $R>0$,
\[
\Big|\iint_{\{|x|\le R\}}\frac{\partial f}{\partial x_1}(x_1,x_2)\,{\rm d}x_1\wedge{\rm d}x_2\Big|=
\Big|\oint_{\{|x|=R\}}f(x_1,x_2)\,{\rm d}x_2\Big|\le 2\pi R\max_{\{|x|=R\}}|f|=O(1/R)\,.
\]
This implies that $\int_{\R^2}(\partial f/\partial x_1)\,{\rm d}x=0$. 
Arguing in the same way as above we also conclude that $\int_{\R^2}(\partial f/\partial x_2)\,{\rm d}x=0$.
Hence, in view of the second statement of Proposition \ref{prop:laplace_operator}, 
the element $\Delta^{-1}(\nabla\circ Q\circ R_{\varphi^{-1}}(v))$ does not have a leading asymptotic term
of the form $c\,\chi(r)\log r$ with $c\ne 0$. This implies that 
$\Delta^{-1}(\nabla\circ Q\circ R_{\varphi^{-1}}(v))\in{\widehat\A}^{m,p}_{1,N+1;-1}\cap\A^{m,p}_{0,N;0}$.
As by Corollary \ref{coro:invariance_1} and Corollary \ref{coro:invariance_2}, the right translation
$R_\varphi$ preserves the space ${\widehat\A}^{m,p}_{1,N+1;-1}\cap\A^{m,p}_{0,N;0}$,
we conclude that
${\mathcal E}_2(\varphi,v)\in{\widehat\A}^{m,p}_{1,N+1;-1}\cap\A^{m,p}_{0,N;0}$.
\endproof

\begin{Rem}
The statement of Lemma \ref{lem:no_log_terms_in_E} can be also easily obtained from the commutative diagram
in Remark \ref{rem:laplace_operator}.
\end{Rem}

In view of Lemma \ref{lem:no_log_terms_in_E}, for $1+\frac{d}{p}<m_0\le m$ consider the following
closed subspace in $\A^{m_0,p}_{0,N;0}$,
\begin{equation}\label{eq:A-tilde}
{\widetilde\A}^{m_0,p}_{1,N;-1}:=\A^{m_0,p}_{1,N;-1}\oplus
\mathop{\rm span}_{\R} \big\{u_k^j\,\big|\, 1\leq j\leq s_k, \ \max\{1,d-2\}\le k\le N\big\},
\end{equation}
where the functions $u_k^j$ are defined in \eqref{def:ujk}.  Clearly,
\[
{\widehat\A}^{m_0,p}_{1,N+1;-1}\cap\A^{m_0,p}_{0,N;0}\subseteq{\widetilde\A}^{m_0,p}_{1,N;-1}\,.
\]
In this way we have
\begin{equation}\label{eq:E-map}
{\mathcal E} : \A D^{m,p}_{0,N;0}\times \A^{m,p}_{0,N;0}\to
\A^{m,p}_{0,N;0}\times{\widetilde\A}^{m,p}_{1,N;-1}\subseteq
\A^{m,p}_{0,N;0}\times\A^{m,p}_{0,N;0},
\end{equation}
and hence represents a vector field on 
$T\big(\A D^{m,p}_{0,N;0}\big)\equiv\A D^{m,p}_{0,N;0}\times \A^{m,p}_{0,N;0}$.
In Section \ref{sec:smoothness_of_the _vector_field} we prove that \eqref{eq:E-map} is real-analytic
(see Theorem \ref{th:E-smooth}).

%%%%%%%%%%%%%%%%%%%%%%%%%%%%%%%%%%
\section{Smoothness of $\Delta_\varphi$ and its inverse}\label{sec:conjugate_laplace_operator}
In this section we study the smoothness of the conjugate Laplace operator \eqref{eq:conjugate_laplacian} and its inverse
(that appears as the first factorization term in formula \eqref{eq:E-factoring}) as a function of
$\varphi\in\A D^{m,p}_{0,N;0}$.
This is used in Section \ref{sec:smoothness_of_the _vector_field} for the proof of 
the smoothness of the Euler vector field. 

\medskip

First, we prove

\begin{Lem}\label{lem:conjugate_laplace_operator1}
For any $(\varphi,v)\in\A D^{m,p}_{0,N;0}\times\A^{m,p}_{1,N+1;0}$,
we have $\Delta_\varphi(v)\in\A^{m-2,p}_{3,N+3;-2}$, and the map
\[
(\varphi,v)\mapsto \Delta_\varphi(v),\quad \A D^{m,p}_{0,N;0}\times\A^{m,p}_{1,N+1;0}\to
\A^{m-2,p}_{3,N+3;-2},
\]
is real-analytic. 
\end{Lem}

\begin{Rem}\label{rem:real-analytic_structure}
As the asymptotic group $\A D^{m,p}_{0,N;0}$ is identified with an open set in
$\A^{m,p}_{0,N;0}$, it is a {\em real-analytic} manifold modeled on $\A^{m,p}_{0,N;0}$.
This allows us to talk about real-analytic maps on $\A D^{m,p}_{0,N;0}$ and its tangent bundle
(cf. \cite{KLT1}).
\end{Rem}

\begin{proof}
As $m> 3+\frac{d}{p}$, in view of the Sobolev embedding we have $\A^{m,p}_{1,N+1;0}\subseteq C^2$. 
A direct computation shows that for any vector field $v\in C^1$ one has
\begin{equation}\label{eq:conjugate_nabla}
(R_\varphi\circ\nabla\circ R_{\varphi^{-1}})(v)=[{\rm d}v]\cdot[{\rm d}\varphi]^{-1}
\end{equation}
and 
\begin{equation}\label{eq:conjugate_div}
(R_\varphi\circ\Div\circ R_{\varphi^{-1}})(v)=\tr\Big([{\rm d}v]\cdot[{\rm d}\varphi]^{-1}\Big).
\end{equation}
It follows from \eqref{eq:conjugate_nabla}, \eqref{eq:conjugate_div},  
and Lemma \ref{lem:conjugate_nabla} below that
\begin{equation}\label{eq:conjugate_nabla'}
(\varphi,v)\mapsto \big(\varphi, (R_\varphi\circ\nabla\circ R_{\varphi^{-1}})(v)\big)\,,
\quad \A D^{m,p}_{0,N;0}\times\A^{m,p}_{1,N+1;0}\to\A^{m,p}_{0,N;0}\times\A^{m-1,p}_{2,N+2;-1},
\end{equation}
and
\begin{equation}\label{eq:conjugate_div'}
(\varphi,v)\mapsto (R_\varphi\circ\Div\circ R_{\varphi^{-1}})(v)\,,
\quad \A D^{m,p}_{0,N;0}\times\A^{m-1,p}_{2,N+2;-1}\to\A^{m-2,p}_{3,N+3;-2}
\end{equation}
are real-analytic maps. Finally, by composing \eqref{eq:conjugate_nabla'} and \eqref{eq:conjugate_div'} and then,
using $\Delta=\Div\circ\nabla$, we prove the statement of the Lemma.
\end{proof}

\medskip

\noindent In the proof of Lemma \ref{lem:conjugate_laplace_operator1} we used the following

\begin{Lem}\label{lem:conjugate_nabla}
Assume that $1+\frac{d}{p}<m_0\le m$, $n\ge 0$, $l\ge -n$. Then the map,
\begin{equation}\label{eq:conjugate_nabla_map}
(\varphi,v)\mapsto (R_\varphi\circ\nabla\circ R_{\varphi^{-1}})(v),\quad
\A D^{m,p}_{0,N,0}\times\A^{m_0,p}_{n,N+n;\ell}\to\A^{m_0-1,p}_{n+1,N+n+1;\ell-1}\,,
\end{equation}
is real-analytic.
\end{Lem}

\begin{proof}
In view of Lemma \ref{lem:inverse_jacobian} in Appendix \ref{sec:appendix_properties} we have 
$[{\rm d}\varphi]^{-1}=\Id+T(\varphi)$ where
$\varphi\mapsto T(\varphi)$, $\A D^{m,p}_{0,N;0}\to\A^{m,p}_{1,N+1;-1}$, is a real-analytic map.
For any $v\in\A^{m_0,p}_{n,N+n;\ell}$ we have that $[{\rm d}v]\in\A^{m_0-1,p}_{n+1,N+n+1;\ell-1}$ by 
Lemma \ref{lem:properties_log_spaces} $(c)$, and
\[
[{\rm d}v]\cdot[{\rm d}\varphi]^{-1}=[{\rm d}v]+[{\rm d}v]\cdot T(\varphi)\,.
\]
In view of  Lemma \ref{lem:properties_log_spaces} $(d)$, 
\[
[{\rm d}v]\cdot T(\varphi)\in\A^{m_0-1,p}_{n+2,N^-+n+2;\ell-2}\subseteq \A^{m_0-1,p}_{n+1,N+n+1;\ell-1}\,,
\]
where we set $N^-=N-1$. Hence, $[{\rm d}v]\cdot[{\rm d}\varphi]^{-1}\in\A^{m-1,p}_{2,N+2;-1}$.
In view of the boundedness of the involved multi-linear maps we see that \eqref{eq:conjugate_nabla_map}
is real analytic. 
\end{proof}

\medskip

\noindent In view of Lemma \ref{lem:conjugate_laplace_operator1} and Corollary \ref{coro:invariance_2}, 
we have

\begin{Coro}\label{coro:conjugate_laplace_operator}
For any $\varphi\in\A D^{m,p}_{0,N;0}$, we have
$\Delta_\varphi({\widehat\A}^{m,p}_{1,N+1;-1})\subseteq\A^{m-2,p}_{3,N+3;-3}$, and the map
\[
(\varphi,v)\mapsto \Delta_\varphi(v),\quad\A D^{m,p}_{0,N;0}\times {\widehat\A}^{m,p}_{1,N+1;-1}\to
\A^{m-2,p}_{3,N+3;-3}
\]
is real-analytic.
\end{Coro}

\begin{proof}
In view of Corollary \ref{coro:conjugation_isomorphism},
\[
(\varphi,v)\mapsto \Delta_\varphi(v),\quad \A D^{m,p}_{0,N;0}\times {\widehat\A}^{m,p}_{1,N+1;-1}\to
\A^{m-2,p}_{3,N+3;-3}\subseteq\A^{m-2,p}_{3,N+3;-2},
\]
where $\A^{m-2,p}_{3,N+3;-3}$ is a closed subspace in $\A^{m-2,p}_{3,N+3;-2}$. 
The statement of the Corollary then follows as, by Lemma \ref{lem:conjugate_laplace_operator1}, the map
\[
(\varphi,v)\mapsto \Delta_\varphi(v),\quad \A D^{m,p}_{0,N;0}\times\A^{m,p}_{1,N+1;0}\to
\A^{m-2,p}_{3,N+3;-2},
\]
is real-analytic and the inclusion ${\widehat A}^{m,p}_{1,N+1;-1}\subseteq\A^{m,p}_{1,N+1;0}$ is bounded.
\end{proof}

Now, consider the map
\begin{equation}\label{eq:D}
{\mathcal D} : \A D^{m,p}_{0,N;0}\times \A^{m-2,p}_{3,N+3;-3}\to{\widehat\A}^{m,p}_{1,N+1;-1},
\quad (\varphi,v)\stackrel{\mathcal D}{\mapsto} (R_\varphi\circ\Delta^{-1}\circ R_{\varphi^{-1}})(v)\,.
\end{equation}
We have
\begin{Prop}\label{prop:conjugate_laplace_operator_inverse}
The map \eqref{eq:D} is real-analytic.
\end{Prop}
\begin{proof}
Denote $E:={\widehat\A}^{m,p}_{1,N+1;-1}$ and $F:=\A^{m-2,p}_{3,N+3;-3}$.
Let ${\rm GL}(E,F)$ be the set of linear isomorphisms $G : E\to F$ inside the Banach space ${\mathcal L}(E,F)$
of bounded linear maps $E\to F$ equipped with the uniform norm. Using Neumann series one easily sees that
${\rm GL}(E,F)$ is an open set in ${\mathcal L}(E,F)$ and that the map
\begin{equation}\label{eq:GL-inverse}
G\to G^{-1},\quad {\rm GL}(E,F)\to{\rm GL}(F,E)
\end{equation}
is real-analytic.
It follows from Corollary \ref{coro:conjugation_isomorphism} that the map
\begin{equation}\label{eq:GL_map}
\varphi\mapsto\Delta_\varphi,\quad \A D^{m,p}_{0,N;0}\to{\rm GL}(E,F)\subseteq{\mathcal L}(E,F)
\end{equation}
is well-defined. Note that $\Delta_\varphi=(D_2{\mathcal C})(\varphi,0)$ where 
${\mathcal C}(\varphi,v):=\Delta_\varphi(v)$ is the map in Corollary \ref{coro:conjugate_laplace_operator}
and $D_2$ denotes the partial derivative of this map with respect to the second argument at the point $(\varphi,0)$. 
This together with Corollary \ref{coro:conjugate_laplace_operator} then implies that
the map \eqref{eq:GL_map} is real-analytic. Composing the map \eqref{eq:GL_map} with
\eqref{eq:GL-inverse} and using that $(\Delta_\varphi)^{-1}=R_\varphi\circ\Delta^{-1}\circ R_{\varphi^{-1}}$
we conclude that the map
\[
\varphi\mapsto R_\varphi\circ\Delta^{-1}\circ R_{\varphi^{-1}},\quad
\A D^{m,p}_{0,N;0}\to{\rm GL}(F,E)\subseteq{\mathcal L}(F,E)
\]
is real-analytic. This completes the proof of the Proposition.
\end{proof}

%%%%%%%%%%%%%%%%%%%%%%%%%%%%%%%%%%
\section{Smoothness of the Euler vector field}\label{sec:smoothness_of_the _vector_field}
In this section we prove Theorem \ref{th:E-smooth}, which states that the Euler vector field defined in
Section \ref{sec:euler_vector_field} is real-analytic.

\smallskip

It follows from Lemma \ref{lem:properties_log_spaces} $(c)$ and $(d')$ that the map
\[
Q : u\mapsto \tr \big([{\rm d}u]^2\big),\quad\A^{m,p}_{0,N;0}\to\A^{m-1,p}_{2,N+2;-2}
\]
is real-analytic.  Now we prove the following.

\begin{Lem}\label{lem:Q_smooth}
For any $(\varphi,v)\in\A D^{m,p}_{0,N;0}\times\A^{m,p}_{0,N;0}$ we have
$(R_\varphi\circ Q\circ R_{\varphi^{-1}})(v)\in\A^{m-1,p}_{2,N+2;-2}$ and the map
\[
(\varphi,v)\mapsto (R_\varphi\circ Q\circ R_{\varphi^{-1}})(v),\quad
\A D^{m,p}_{0,N;0}\times\A^{m,p}_{0,N;0}\to\A^{m-1,p}_{2,N+2;-2}
\]
is real-analytic.
\end{Lem}
\begin{proof}
In view of the Sobolev embedding, we have $\A^{m,p}_{0,N;0}\subseteq C^2$. 
For any vector field $v\in\A^{m,p}_{0,N;0}$ and for any $\varphi\in\A D^{m,p}_{0,N;0}$ we have
\[
(R_\varphi\circ Q\circ R_{\varphi^{-1}})(v)=
R_\varphi\circ\tr\left(\Big[(\nabla\circ R_{\varphi^{-1}})(v)\Big]^2\right)
=\tr\left(\Big[(R_\varphi\circ\nabla\circ R_{\varphi^{-1}})(v)\Big]^2\right)\,.
\]
This, together with Lemma \ref{lem:conjugate_nabla} and Lemma \ref{lem:properties_log_spaces} $(d')$ in
Appendix \ref{sec:appendix_properties}, then completes the proof of the Lemma.
\end{proof}

\noindent

Now, consider the map
\begin{equation}\label{eq:B}
{\mathcal B} : \A D^{m,p}_{0,N;0}\times\A^{m,p}_{0,N;0}\to\A D^{m,p}_{0,N;0}\times\A^{m-2,p}_{3,N+3;-3},
\quad (\varphi,v)\stackrel{\mathcal B}{\mapsto}\big(\varphi,(R_\varphi\circ\nabla\circ Q\circ R_{\varphi^{-1}})(v)\big),
\end{equation}
and note that, by Corollary \ref{coro:invariance_1}, 
$(R_\varphi\circ\nabla\circ Q\circ R_{\varphi^{-1}})(v)=
(R_\varphi\circ\nabla\circ R_{\varphi^{-1}})\circ(R_{\varphi}\circ Q\circ R_{\varphi^{-1}})(v)$.
It follows from Lemma \ref{lem:conjugate_nabla} that the map
\[
(\varphi,v)\mapsto (R_\varphi\circ\nabla\circ R_{\varphi^{-1}})(v),\quad
\A D^{m,p}_{0,N;0}\times\A^{m-1,p}_{2,N+2;-2}\to\A^{m-2,p}_{3,N+3;-3}
\]
is real-analytic. Combining this with Lemma \ref{lem:Q_smooth} we get
\begin{Prop}\label{prop:B-smooth}
The map \eqref{eq:B} is real-analytic.
\end{Prop}

\noindent By \eqref{eq:E-factoring} the Euler vector field \eqref{eq:E} defined 
in Section \ref{sec:euler_vector_field} is decomposed as
\begin{equation*}
{\mathcal E}(\varphi,v)=\big(v,({\mathcal D}\circ{\mathcal B})(\varphi,v)\big)\,.
\end{equation*}
This together with Proposition \ref{prop:conjugate_laplace_operator_inverse}, 
Proposition \ref{prop:B-smooth}, and Lemma \ref{lem:no_log_terms_in_E}, then implies

\begin{Coro}\label{coro:E-smooth}
The map
\[
{\mathcal E} : \A D^{m,p}_{0,N;0}\times \A^{m,p}_{0,N;0}\to\A^{m,p}_{0,N;0}\times
{\widetilde\A}^{m,p}_{1,N;-1}
\]
is real-analytic. (We refer to \eqref{eq:A-tilde} for the definition of ${\widetilde\A}^{m,p}_{1,N;-1}$.)
\end{Coro}

\noindent In view of the boundedness of ${\widetilde\A}^{m,p}_{1,N;-1}\subseteq\A^{m,p}_{0,N;0}$ we finally
have

\begin{Th}\label{th:E-smooth} The Euler vector field
\[
{\mathcal E} : \A D^{m,p}_{0,N;0}\times \A^{m,p}_{0,N;0}\to\A^{m,p}_{0,N;0}\times 
\A^{m,p}_{0,N;0}
\]
is real-analytic.
\end{Th}

%%%%%%%%%%%%%%%%%%%%%%
\section{Proof of Theorem \ref{th:main} and Proposition \ref{prop:integrals}}\label{sec:proof_main_theorem}
Assume that $u\in C^0\big([0,T],\A^{m,p}_{0,N;0}\big)\cap C^1\big([0,T],\A^{m-1,p}_{0,N;0}\big)$
is a solution of equation \eqref{eq:euler'}. Then, in view of Proposition \ref{prop:ode},
there exists a unique solution $\varphi\in C^1\big([0,T],\A D^{m,p}_{0,N;0}\big)$, of the 
equation $\dt\varphi=u\circ\varphi$, $\varphi|_{t=0}=\id$. Denote $v:=\dt\varphi$. Then, 
\begin{equation}\label{eq:the_integral_curve}
(\varphi,v)\in C^0\big([0,T], T(\A D^{m,p}_{0,N;0})\big)\,.
\end{equation}
We will prove that, in fact, this curve is $C^1$-smooth and that it is the integral curve of the
Euler vector field $\mathcal E$ on $T(\A D^{m,p}_{0,N;0})$: In view of the Sobolev embedding
$\A^{m-1,p}_{0,N;0}\subseteq C^1$ and $u\in C^1\big([0,T],{\A}^{m-1,p}_{0,N;0}\big)$ 
we conclude that $u,\,\varphi\in C^1\big([0,T]\times\R^d,\R^d\big)$\,.
Since $v(t,x)=u(t,\varphi(t,x))$, we have 
\[
v_t=u_t\circ\varphi+[{\rm d}u]\circ\varphi\cdot{\dt\varphi}=\big(u_t+u\cdot\nabla u\big)\circ\varphi,
\]
where $v_t$ stands for the point-wise partial derivative $v_t(t,x)$. In view of \eqref{eq:euler'} and
\eqref{eq:E}, we then obtain
\[
v_t=\big(R_\varphi\circ\Delta^{-1}\circ\nabla\circ Q\circ R_{\varphi^{-1}}\big)(v)=
{\mathcal E}_2(\varphi,v)\,.
\]
By a point-wise integration with respect to $t$ we then have
\begin{equation}\label{eq:integral_formula}
v(t,x)=u_0(x)+\int_0^t{\mathcal E}_2(\varphi,v)|_{(s,x)}\,{\rm d}s\,.
\end{equation}
Now, it follows from \eqref{eq:the_integral_curve} and Theorem \ref{th:E-smooth} that
the curve $s\mapsto{\mathcal E}_2(\varphi(s),v(s))$, $[0,T]\to\A^{m,p}_{0,N;0}$, is continuous.
Hence, the integral in \eqref{eq:integral_formula} converges in $\A^{m,p}_{0,N;0}$
and $v(t)=u_0+\int_0^t{\mathcal E}_2(\varphi(s),v(s))\,{\rm d}s$ holds in $\A^{m,p}_{0,N;0}$.
This implies that $v\in C^1\big([0,T],\A^{m,p}_{0,N;0}\big)$ and ${\dt v}={\mathcal E}_2(\varphi,v)$.
Hence, $(\varphi,v)\in C^1\big([0,T], T(\A D^{m,p}_{0,N;0})\big)$ is the integral curve of the
dynamical system,
\begin{equation}\label{eq:the_dynamical_system}
\left\{
\begin{array}{l}
(\dt\varphi,\dt v)=\big(v,{\mathcal E}_2(\varphi,v)\big)\equiv{\mathcal E}(\varphi,v),\\
(\varphi,v)|_{t=0}=(\id,u_0),\,\,\,u_0\in\A^{m,p}_{0,N;0}\,.
\end{array}
\right.
\end{equation}
Conversely, one easily confirms, by using Lemma \ref{lem:no_pressure_term} and the fact that 
the solutions $(\varphi,v)$ of \eqref{eq:the_dynamical_system} satisfy $v=\dt\varphi$,
that any solution $(\varphi,v)\in C^1\big([0,T], T(\A D^{m,p}_{0,N;0})\big)$ of \eqref{eq:the_dynamical_system}
produces a solution $u:=R_{\varphi^{-1}}v$ of equation \eqref{eq:euler'}.
Note that, by Theorem \ref{th:the_group}, we have 
$u\in C^0\big([0,T],\A^{m,p}_{0,N;0}\big)\cap C^1\big([0,T],\A^{m-1,p}_{0,N;0}\big)$.
In this way, we proved

\begin{Lem}\label{lem:dynamical_system<-->euler'}
The map
\[
(\varphi,v)\mapsto R_{\varphi^{-1}}v,\quad
C^1\big([0,T], T(\A D^{m,p}_{0,N;0})\big)\to
C^0\big([0,T],\A^{m,p}_{0,N;0}\big)\cap 
C^1\big([0,T],\A^{m-1,p}_{0,N;0}\big)
\]
gives a bijective correspondence between solutions of the dynamical system \eqref{eq:the_dynamical_system}
and the solutions of equation \eqref{eq:euler'} in 
$C^0\big([0,T],\A^{m,p}_{0,N;0}\big)\cap C^1\big([0,T],\A^{m-1,p}_{0,N;0}\big)$.
\end{Lem}

\noindent Combining Lemma \ref{lem:dynamical_system<-->euler'} with Lemma \ref{lem:no_pressure_term} we have

\begin{Prop}\label{prop:dynamical_system}
The map
\[
(\varphi,v)\mapsto R_{\varphi^{-1}}v,\quad
C^1\big([0,T], T(\A D^{m,p}_{0,N;0})\big)\to
C^0\big([0,T],\accentset{\,\,\,\circ}{\A}^{m,p}_{0,N;0}\big)\cap 
C^1\big([0,T],\accentset{\,\,\,\circ}{\A}^{m-1,p}_{0,N;0}\big)
\]
gives a bijective correspondence between solutions of the dynamical system \eqref{eq:the_dynamical_system}
with $u_0\in\accentset{\,\,\,\circ}{\A}^{m,p}_{0,N;0}$ and the solutions of the Euler equation \eqref{eq:euler}
in 
$C^0\big([0,T],\accentset{\,\,\,\circ}{\A}^{m,p}_{0,N;0}\big)\cap 
C^1\big([0,T],\accentset{\,\,\,\circ}{\A}^{m-1,p}_{0,N;0}\big)$.
\end{Prop}

\medskip

\noindent{\em Proof of Theorem \ref{th:main}.}
The Theorem follows directly from Theorem \ref{th:E-smooth}, Proposition \ref{prop:dynamical_system}, 
and the existence and uniqueness theorem of solutions of an ODE in a Banach space (\cite{Lang}).
We also use that the solutions of \eqref{eq:euler'} (and hence \eqref{eq:euler}) admit the symmetry group
$u\mapsto u_c$ where $u_c(t):=c\,u(c t)$, $c>0$.
\finishproof

\medskip

\noindent{\em Proof of Proposition \ref{prop:integrals}.}
Let $u\in C^0\big([0,T],\A^{m,p}_{0,N;0}\big)\cap 
C^1\big([0,T],\A^{m-1,p}_{0,N;0}\big)$
be the solution of the Euler equation \eqref{eq:euler} given by Theorem \ref{th:main}.
Then, by Lemma \ref{lem:no_pressure_term}, $u$ is a solution of
\begin{equation}\label{eq:euler''}
u_t+u\cdot\nabla u=\Delta^{-1}\circ\nabla\circ Q(u),\quad u|_{t=0}=u_0\,.
\end{equation}
Note that by Lemma \ref{lem:no_log_terms_in_E}, Corollary \ref{coro:invariance_1} and 
Corollary \ref{coro:invariance_2} the right hand side of
\eqref{eq:euler''} belongs to the closed subspace ${\widetilde\A}^{m,p}_{1,N;-1}$ in $\A^{m,p}_{0,N;0}$
defined in \eqref{eq:A-tilde}. 
Moreover, by Lemma \ref{lem:properties_log_spaces} and Proposition \ref{prop:laplace_operator},
the map $t\mapsto\Delta^{-1}\circ\nabla\circ Q(u(t))$,
$[0,T]\to{\widetilde\A}^{m,p}_{1,N;-1}$, is continuous with respect to the norm in $\A^{m,p}_{0,N;0}$. 
Similarly, by Lemma \ref{lem:properties_log_spaces}, the composed map
$t\mapsto u(t)\cdot\nabla u(t)$, $[0,T]\to\A^{m-1,p}_{1,N+1;-1}\subseteq\A^{m-1,p}_{1,N;-1}$, is continuous
and the space $\A^{m-1,p}_{1,N;-1}$ is closed in $\A^{m-1,p}_{0,N;0}$.
By integrating \eqref{eq:euler''} in $\A^{m-1,p}_{0,N;0}$ and by using that ${\widetilde\A}^{m-1,p}_{1,N;-1}$
is closed in $\A^{m-1,p}_{0,N;0}$ we conclude that for any $t\in[0,T]$ we have
$u(t)-u_0\in{\widetilde\A}^{m-1,p}_{1,N;-1}$. 
As $u\in C^0\big([0,T],\A^{m,p}_{0,N;0}\big)$ we then conclude from 
\eqref{eq:A-tilde} that for any $t\in[0,T]$,
\begin{equation}\label{eq:u(t)-u_0}
u(t)-u_0\in\A^{m,p}_{1,N;-1}\oplus
\mathop{\rm span}_{\R} \big\{u_k^j\,\big|\, 1\le j\le s_k, \ \max\{1,d-2\}\le k\leq N\big\}
\end{equation}
where the functions $u_k^j$ are defined in \eqref{def:ujk}. 
The statement of the Proposition \ref{prop:integrals} now follows directly from \eqref{eq:u(t)-u_0} and \eqref{def:ujk}.
A final remark needed is that the space of homogeneous harmonic polynomials of degree $l\ge 0$ (when restricted to 
the sphere $\s^{d-1}$) coincides with the eigenspace of the ``positive'' Laplace-Beltrami operator $-\Delta_S$ with
eigenvalue $\mu_l=l(k+d-2)$ (see e.g.\ \cite{Shubin}).
\finishproof

\medskip

\noindent{\em Non-triviality of the conserved quantities.}
Let us now discuss the non-triviality of the conservation laws \eqref{eq:integrals} appearing in
Proposition \ref{prop:integrals}. For simplicity, we consider the case $d=2$. 
Take $0\leq k\leq N$ and consider the Hamiltonian $H:=a(\phi)(\log r)^k\chi(r)/r^{k-1}$, where $a$ is a 
$C^\infty$-smooth, real-valued function of $\phi\in\R/2\pi\Z$. 
Let $u$ be the Hamiltonian vector field corresponding to $H$:
\[
\begin{aligned}
u=\big(-H_y,H_x\big)
=\chi(r)\left(\frac{a_k^{k-1}(\phi)\,(\log r)^{k-1}+a_k^k(\phi)\,(\log r)^k}{r^k}\right)+f,
\end{aligned}
\]
where $f$ has compact support and
\[
a_k^k(\phi)=\Big( -a'(\phi)\cos\phi+(k-1)a(\phi)\sin\phi,-a'(\phi)\sin\phi-(k-1)a(\phi)\cos\phi\Big),
\]
with $a'=da/d\phi$. Clearly, $u\in\accentset{\,\,\,\circ}\A^{m,p}_{N;0}$. 
Now, if we identify real 2-vectors with complex numbers
and denote by $\sum_{l\in\Z}{\hat a}_l\,e^{-l\phi}$ the Fourier series of $a(\phi)$, we obtain
\[
\begin{aligned}
a_k^k(\phi)&= -a'(\phi)\cos\phi+(k-1)a(\phi)\sin\phi -i(a'(\phi)\sin\phi+(k-1)\cos\phi) \\
&=-\Big(a'(\phi)+i(k-1)a(\phi)\Big)e^{i\phi}\\
&=-i\sum_{l\in\Z}(l+k-2){\hat a}_{l-1}\,e^{i l\phi}.
\end{aligned}
\]
In particular, we see that if $h$ is a homogeneous harmonic polynomial of degree $l\notin\{2-k,k\}$ then the integral 
\eqref{eq:integrals} is a non-vanishing linear function of the Fourier coefficients ${\hat a}_{l-1}$ and 
${\hat a}_{-l-1}$ of $a(\phi)$.

%%%%%%%%%%%%%%%%%%%%%%%%%%%%%%%%%%
\section{Asymptotic groups and asymptotic ODEs}\label{sec:ode}
In this section we prove Proposition \ref{prop:ode} stated in Section \ref{sec:the_group}.

\medskip

\noindent{\em Proof of Proposition \ref{prop:ode}.}
Assume that $0\le n\le N$, $-n\le \ell\le 0$.
Take $u\in C^0\big([0,T],\A^{m,p}_{n,N;\ell}\big)$ and consider the differential equation,
\begin{equation}\label{eq:ode2}
\dt\varphi(t)=u(t)\circ\varphi(t)\,.
\end{equation}
First, note that in view of Theorem \ref{th:the_group},
the map $g(t,\varphi):=u(t)\circ\varphi$, $g : [0,T]\times\A D^{m-1,p}_{n,N;\ell}\to\A^{m-1,p}_{n,N;\ell}$
and its partial derivative with respect to the second argument 
$D_2g :  [0,T]\times\A D^{m-1,p}_{n,N;\ell}\to{\mathcal L}(\A^{m-1,p}_{n,N;\ell},\A^{m-1,p}_{n,N;\ell})$
are continuous maps. This implies that $g$ is locally Lipschitz with respect to the second argument 
in $[0,T]\times\A D^{m-1,p}_{n,N;\ell}$.
In particular, by the existence and uniqueness theorem of solutions of an ODE in a Banach space, for any given
$\tau\in[0,T]$ there exist an open neighborhood $U(\tau)$ of $\tau$ in $[0,T]$
and an open neighborhood $V(\id)$ of the identity $\id$ in $\A D^{m-1,p}_{n,N;\ell}$ so that for any 
$t_0\in U(\tau)$ there exists a unique solution $\varphi\in C^1\big(U(\tau),\A D^{m-1,p}_{n,N;\ell}\big)$
of \eqref{eq:ode2} with initial condition $\varphi|_{t=t_0}=\id$.
Now, using the invariance of \eqref{eq:ode2} with respect to right-translations as well as the fact that
the right translation $R_{\varphi_0} : \A D^{m-1,p}_{n,N;\ell}\to\A D^{m-1,p}_{n,N;\ell}$ for a fixed element
$\varphi_0\in\A D^{m-1,p}_{n,N;\ell}$ is a $C^\infty$-smooth map, we see that for any $t_0\in U(\tau)$ and
for any $\varphi_0\in\A D^{m-1,p}_{n,N;\ell}$, equation \eqref{eq:ode2} with initial condition 
$\varphi|_{t=t_0}=\varphi_0$ has a unique solution $\varphi\in C^1\big(U(\tau),\A D^{m-1,p}_{n,N;\ell}\big)$.
This, together with the compactness of the interval $[0,T]$, implies that \eqref{eq:ode2} with initial condition
$\varphi|_{t=0}=\id$ has a unique solution,
\begin{equation}\label{eq:varphi_rough}
\varphi\in C^1\big([0,T],\A D^{m-1,p}_{n,N;\ell}\big).
\end{equation}
Further, we will use a bootstrap argument to show that this solution in fact lies in the space
$C^1\big([0,T],\A D^{m,p}_{n,N;\ell}\big)$.
As $\varphi\in C^1\big([0,T],\A D^{m-1,p}_{n,N;\ell}\big)$ satisfies \eqref{eq:ode2} with initial condition
$\varphi|_{t=0}=\id$, we obtain by applying $\partial/\partial x_k$, $k=1,...,d$, to \eqref{eq:ode2},
\begin{equation}\label{eq:linear_ode}
\left\{
\begin{array}{l}
{[{\rm d}\varphi]}^{\cdt}=[{\rm d}u]\circ\varphi\cdot[{\rm d}\varphi],\cr
[{\rm d}\varphi]_{t=0}=\Id,
\end{array}
\right.
\end{equation}
where by Theorem \ref{th:the_group},
\begin{equation}\label{eq:the_matrix}
A:=[{\rm d}u]\circ\varphi\in C^0\big([0,T],\A^{m-1,p}_{n+1,N+1;\ell-1}\big),
\end{equation}
and
\begin{equation}\label{eq:the_linear_solution}
[{\rm d}w]=[{\rm d}\varphi]-\Id\in C^1\big([0,T],\A^{m-2,p}_{n+1,N+1;\ell-1}\big),
\end{equation}
where $\varphi(x)=x+w(x)$, $w\in\A^{m-1,p}_{n,N;\ell}$.
In view of \eqref{eq:the_matrix} and the Banach algebra property of the
space $\A^{m-1,p}_{n+1,N+1;\ell-1}$, the equation ${[v]}^{\cdt}=A(t)\cdot [v]+A(t)$ with initial condition 
$[v]|_{t=0}=0$ is a linear ODE in the Banach space $\A^{m-1,p}_{n+1,N+1;\ell-1}$. This implies that there exists a 
unique solution $[v]\in C^1\big([0,T],\A^{m-1,p}_{n+1,N+1;\ell-1}\big)$. This together with \eqref{eq:linear_ode} and
\eqref{eq:the_linear_solution} then implies that,
\begin{equation}\label{eq:d_xw}
[{\rm d}w]\in C^1\big([0,T],\A^{m-1,p}_{n+1,N+1;\ell-1}\big)\,.
\end{equation}
Applying the divergence operator $\Div$ to the rows of this matrix we see that
\[
\Delta w \in C^1\big([0,T],\A^{m-2,p}_{n+2,N+2;\ell-2}\big)\,.
\]
This together with \cite[Proposition B.4]{McOwenTopalov2} and Remark \ref{rem:laplace_operator}
implies that $w\in C^1\big([0,T],\A^{m,p}_{0,N;\ell+2}\big)$. 
Comparing this with \eqref{eq:varphi_rough} we finally see that 
$w\in C^1\big([0,T],\A^{m,p}_{n,N;\ell}\big)$. This completes the proof of the
Proposition.
\finishproof

\medskip

\noindent{\em Proof of Corollary \ref{coro:ode_volume_preserving}.}
Let ${\rm M}_{d\times d}$ be the linear space of $d\times d$ matrices.
In view of the Sobolev embedding $\A^{m-1,p}_{n+1,N+1;\ell-1}\subseteq C^0$,
we see from \eqref{eq:linear_ode} and \eqref{eq:d_xw} that for any given $x\in\R^d$,
$[{\rm d}\varphi](\cdot,x)\in C^1\big([0,T],{\rm M}_{d\times d}\big)$
is the fundamental matrix of the linear system $\dt X=A(t,x) X$, $X|_{t=0}=\Id$, where 
$A(t,x):=[{\rm d}u]\big|_{(t,\varphi(t,x))}$. Since $\Div u(t)=0$ for any $t\in[0,T]$, we see from the Wronskian identity 
that for any $t\in[0,T]$ we have
\begin{equation}\label{eq:wronskian_identity}
\det[{\rm d}\varphi](t,x)=e^{\int_0^t\tr[{\rm d}u]|_{(s,\varphi(s,x))}\,{\rm d}s}=
e^{\int_0^t(\Div u)|_{(s,\varphi(s,x))}\,{\rm d}s}\,.
\end{equation}
As $\Div u\equiv 0$, we conclude that for any $t\in[0,T]$ we have $\det[{\rm d}\varphi](t,x)=1$.
This, together with  Proposition \ref{prop:ode}, then implies that 
$\varphi\in{\accentset{\,\,\,\circ}{\A}} D^{m,p}_{n,N;\ell}$.
\finishproof

%%%%%%%%%%%%%%%%%%%%%%%%%%%%%%%%%%%
\section{Volume preserving asymptotic diffeomorphisms}\label{sec:volume_preserving_diffeomorphisms}
Recall from Section \ref{sec:the_group} that
\[
\accentset{\,\,\,\circ}{\A} D^{m,p}_{0,N;0}=\big\{\varphi\in\A D^{m,p}_{0,N;0}\,\big|\,
\forall x\in\R^d,\,\det[{\rm d}_x\varphi]=1\big\}
\]
is the subgroup of volume preserving diffeomorphisms in $\A D^{m,p}_{0,N;0}$.
Clearly, $\accentset{\,\,\,\circ}{\A} D^{m,p}_{0,N;0}$ is a closed set in $\A D^{m,p}_{0,N;0}$ and it is
a topological subgroup with respect to the topology inherited from $\A D^{m,p}_{0,N;0}$.
Moreover, by Lemma \ref{lem:properties_log_spaces}, the map 
\[
\varphi\mapsto\det[{\rm d}\varphi]-1,\quad\A D^{m,p}_{0,N;0}\to\A^{m-1,p}_{1,N+1;-1},
\]
is a polynomial map, and hence $\accentset{\,\,\,\circ}{\A} D^{m,p}_{0,N;0}$ is a real analytic subvariety in 
$\A D^{m,p}_{0,N;0}$. We will prove that, in fact, $\accentset{\,\,\,\circ}{\A} D^{m,p}_{0,N;0}$ is a 
{\em real-analytic submanifold} in $\A D^{m,p}_{0,N;0}$.
More specifically, we prove

\begin{Th}\label{th:volume_preserving_diffeomorphisms}
\begin{itemize}
\item[$(a)$] $\accentset{\,\,\,\circ}{\A} D^{m,p}_{0,N;0}$ is a real-analytic submanifold in $\A D^{m,p}_{0,N;0}$;
\item[$(b)$] For any $\varphi\in\accentset{\,\,\,\circ}{\A} D^{m,p}_{0,N;0}$ the tangent space
$T_\varphi\big(\accentset{\,\,\,\circ}{\A} D^{m,p}_{0,N;0}\big)$ to 
$\accentset{\,\,\,\circ}{\A} D^{m,p}_{0,N;0}$ at $\varphi$ coincides with 
$\left(\varphi,R_\varphi\big(\accentset{\,\,\,\circ}{\A}^{m,p}_{0,N;0}\big)\right)$ where 
$R_\varphi\big(\accentset{\,\,\,\circ}{\A}^{m,p}_{0,N;0}\big)$ is the image of the right-translation
$R_\varphi : \accentset{\,\,\,\circ}{\A}^{m,p}_{0,N;0}\to\A^{m,p}_{0,N;0}$;
\item[$(c)$] The Euler vector field $\mathcal E$ is tangent to the submanifold 
$T(\accentset{\,\,\,\circ}{\A}D^{m,p}_{0,N;0})$ in $T(\A D^{m,p}_{0,N;0})$.
\end{itemize}
\end{Th}

Before proving this Theorem let us first discuss the {\em exponential map} of the spray $\mathcal E$.
As ${\mathcal E}$ is a real analytic vector field on $T\big(\A D^{m,p}_{0,N;0}\big)$, we see from
the existence and uniqueness theorem of solutions of an analytic ODE in a Banach space (see e.g
\cite[Theorem 10.8.1, Theorem 10.8.2]{Dieudonne}) that there exists an open neighborhood $\mathcal U$
of zero in ${\A}^{m,p}_{0,N;0}$ such that for any initial data $(\varphi,v)|_{t=0}=(\id,u_0)$ the ODE
\begin{equation}\label{eq:DS}
(\dt\varphi,\dt v)={\mathcal E}(\varphi,v)\equiv\big(v,{\mathcal E}_2(\varphi,v)\big)
\end{equation}
has a unique real-analytic solution $(-2,2)\to\A D^{m,p}_{0,N;0}\times\A^{m,p}_{0,N;0}$,
$t\mapsto\big(\varphi(t;u_0),v(t;u_0)\big)$, so that the map
\begin{equation}\label{eq:DS-solution}
(-2,2)\times{\mathcal U}\to\A D^{m,p}_{0,N;0}\times\A^{m,p}_{0,N;0},\quad
(t,u_0)\mapsto\big(\varphi(t;u_0),v(t;u_0)\big),
\end{equation}
is real analytic. Here we use that the solutions of \eqref{eq:DS} admit the symmetry 
group $(\varphi,v)\mapsto(\varphi_c,v_c)$ where
$\varphi_c(t):=\varphi(c t)$ and $v_c(t):=c\,v(c t)$, $c>0$.
Now, consider the exponential map of $\mathcal E$,
\begin{equation}\label{eq:exponential_map}
\text{\rm Exp}_{\mathcal E} : {\mathcal U}\to\A D^{m,p}_{0,N;0},\quad u_0\mapsto\varphi(1;u_0),
\end{equation}
that, by the discussion above, is real analytic. By the discussed symmetry, for any $t\in(-2,2)$ one has
$\text{\rm Exp}_{\mathcal E}(t\,u_0)=\varphi(t)$.
This implies that ${\rm d}_0\text{\rm Exp}_{\mathcal E}=\id_{\A^{m,p}_{0,N;0}}$. 
Shrinking the neighborhood  $\mathcal U$ if necessary, we obtain from the inverse function theorem
the following

\begin{Prop}\label{prop:exp_local_diffeomorphism}
There exists an open neighborhood $\mathcal U$ of zero in ${\A}^{m,p}_{0,N;0}$ and an open
neighborhood $\mathcal V$ of $\id$ in $\A D^{m,p}_{0,N;0}$ so that 
$\text{\rm Exp}_{\mathcal E} : {\mathcal U}\to{\mathcal V}$ is a real analytic diffeomorphism.
\end{Prop}

Further, we prove

\begin{Prop}\label{prop:exp_local_diffeomorphism'}
For any $u_0\in{\mathcal U}\cap\accentset{\,\,\,\circ}{\A}^{m,p}_{0,N;0}$ we have 
$\text{\rm Exp}_{\mathcal E}(u_0)\in\accentset{\,\,\,\circ}{\A}D^{m,p}_{0,N;0}$ and the map
\begin{equation}\label{eq:exponential_map_restricted}
\text{\rm Exp}_{\mathcal E} : {\mathcal U}\cap\accentset{\,\,\,\circ}{\A}^{m,p}_{0,N;0}\to
{\mathcal V}\cap\accentset{\,\,\,\circ}{\A}D^{m,p}_{0,N;0}.
\end{equation}
is a real analytic diffeomorphism.
\end{Prop}

\begin{Rem}\label{rem:analyticity_in_t}
In particular, we see that for any $u_0\in{\cal U}\cap\accentset{\,\,\,\circ}{\A}^{m,p}_{0,N;0}$ the map 
$t\mapsto\varphi(t;u_0)=\text{\rm Exp}_{\mathcal E}(t u_0)$, $(-2,2)\to\accentset{\,\,\,\circ}\A D^{m,p}_{N;0}$,
is real analytic (cf. \cite{KLT1,KLT2,Serfati,Shnir}).
\end{Rem}

\noindent{\em Proof of Proposition \ref{prop:exp_local_diffeomorphism'}.}
Take $u_0\in{\mathcal U}\cap\accentset{\,\,\,\circ}{\A}^{m,p}_{0,N;0}$ and let $(\varphi,v)$
be the solution of \eqref{eq:DS} discussed above.
In view of Proposition \ref{prop:dynamical_system}, 
the curve 
$u:=R_{\varphi^{-1}}v$ lies in $C^0\big([0,1],\accentset{\,\,\,\circ}{\A}^{m,p}_{0,N;0}\big)\cap 
C^1\big([0,1],\accentset{\,\,\,\circ}{\A}^{m-1,p}_{0,N;0}\big)$. 
As by \eqref{eq:DS} one has $v=\dt\varphi$, we obtain that $\varphi\in C^1\big([0,1],{\A}D^{m,p}_{0,N;0}\big)$
satisfies $\dt\varphi=u\circ\varphi$, $\varphi|_{t=0}=\id$. Then, Corollary \ref{coro:ode_volume_preserving}
implies that $\varphi\in C^1\big([0,1],\accentset{\,\,\,\circ}{\A}D^{m,p}_{0,N;0}\big)$. 
Hence, $\text{\rm Exp}_{\mathcal E}(u_0)=\varphi(1;u_0)\in\accentset{\,\,\,\circ}{\A}D^{m,p}_{0,N;0}$.
We will prove Proposition \ref{prop:exp_local_diffeomorphism'} by showing that 
\eqref{eq:exponential_map_restricted} is onto.
Due to Proposition \ref{prop:exp_local_diffeomorphism}, it suffices to show that if
$u_0\in{\mathcal U}\setminus\accentset{\,\,\,\circ}{\A}^{m,p}_{0,N;0}$ then 
$\text{\rm Exp}_{\mathcal E}(u_0)\notin\accentset{\,\,\,\circ}{\A}D^{m,p}_{0,N;0}$.
Take $u_0\in{\mathcal U}\setminus\accentset{\,\,\,\circ}{\A}^{m,p}_{0,N;0}$
and let $(\varphi,v)$ be the solution of \eqref{eq:DS}.
By Lemma \ref{lem:dynamical_system<-->euler'}, the curve 
$u:=R_{\varphi}v\in C^0([0,1],{\A}^{m,p}_{0,N;0})\cap C^1([0,1],{\A}^{m-1,p}_{0,N;0})$
satisfies equation \eqref{eq:euler'}. Using that $\Delta^{-1}$ and $\nabla$ commute
(Remark \ref{rem:laplace_operator}) we see that $u$ satisfies
$u_t+u\cdot\nabla u=\nabla\circ\Delta^{-1}\circ Q(u)$.
Applying $\Div$ to the both hand sides of this equality we conclude from \eqref{eq:the_trick1} that 
$(\Div u)_t+u\cdot\nabla(\Div u)=0$.
This implies that for any $x\in\R^d$ and for any $t\in[0,1]$, $(\Div u)(t,\varphi(t,x))=(\Div u_0)(x)$.
Then, in view of the Wronskian identity \eqref{eq:wronskian_identity} we conclude that
$\det[{\rm d}\varphi](t,x)=e^{t(\Div u_0)(x)}$, which implies that 
$\text{\rm Exp}_{\mathcal E}(u_0)\notin\accentset{\,\,\,\circ}{\A}D^{m,p}_{0,N;0}$.
Now, as mentioned above, the ontoness of \eqref{eq:exponential_map_restricted} follows from
Proposition \ref{prop:exp_local_diffeomorphism}.
\finishproof

\noindent{\em Proof of Theorem \ref{th:volume_preserving_diffeomorphisms}.}
$(a)$ In view of  Proposition \ref{prop:exp_local_diffeomorphism} and 
Proposition \ref{prop:exp_local_diffeomorphism'}, ${\mathcal V}\cap\accentset{\,\,\,\circ}{\A}D^{m,p}_{0,N;0}$
is diffeomorphic to ${\mathcal U}\cap\accentset{\,\,\,\circ}{\A}^{m,p}_{0,N;0}$ and hence it is a submanifold
in an open neighborhood of $\id$ in $\A D^{m,p}_{0,N;0}$. The general case then follows
as for any given $\psi\in\accentset{\,\,\,\circ}{\A}D^{m,p}_{0,N;0}$ the right
translation $R_\psi : \accentset{\,\,\,\circ}{\A}D^{m,p}_{0,N;0}\to\accentset{\,\,\,\circ}{\A}D^{m,p}_{0,N;0}$
is a diffeomorphism.\footnote{In fact, the map 
$R_\psi : \accentset{\,\,\,\circ}{\A}D^{m,p}_{0,N;0}\to\accentset{\,\,\,\circ}{\A}D^{m,p}_{0,N;0}$
is affine linear, and hence real analytic.}

\smallskip

\noindent $(b)$ The discussion above implies that 
$T_\id\big(\accentset{\,\,\,\circ}{\A}D^{m,p}_{0,N;0}\big)=\accentset{\,\,\,\circ}{\A}^{m,p}_{0,N;0}$.
As above, the general case follows as for any given $\psi\in\accentset{\,\,\,\circ}{\A}D^{m,p}_{0,N;0}$ the right
translation $R_\psi : \accentset{\,\,\,\circ}{\A}D^{m,p}_{0,N;0}\to\accentset{\,\,\,\circ}{\A}D^{m,p}_{0,N;0}$
is a diffeomorphism.

\smallskip

\noindent $(c)$ For any $u_0\in\accentset{\,\,\,\circ}{\A}D^{m,p}_{0,N;0}$ consider the solution
$(\varphi,v)\in C^1\big([0,1], T(\A D^{m,p}_{0,N;0})\big)$
of \eqref{eq:DS} with initial data $(\varphi,u_0)$. 
Arguing as in the proof of Proposition \ref{prop:exp_local_diffeomorphism'}, we conclude from 
Proposition \ref{prop:dynamical_system} and Corollary \ref{coro:ode_volume_preserving} that 
$\varphi\in C^1\big([0,1],\accentset{\,\,\,\circ}{\A}D^{m,p}_{0,N;0}\big)$.
In particular,
$v=\dt\varphi\in C^0\big([0,1],\accentset{\,\,\,\circ}{\A}^{m,p}_{0,N;0}\big)$.
Recalling that $(\varphi,v)\in C^1\big([0,1], T(\A D^{m,p}_{0,N;0})\big)$ we see that
$(\varphi,v)\in C^1\big([0,1], T(\accentset{\,\,\,\circ}\A D^{m,p}_{0,N;0})\big)$.
This implies that for any $u_0\in\accentset{\,\,\,\circ}{\A}D^{m,p}_{0,N;0}$,
${\mathcal E}(\id,u_0)=\frac{d}{dt}\big|_{t=0}\big(\varphi,v\big)$
is tangent to the submanifold $T(\accentset{\,\,\,\circ}\A D^{m,p}_{0,N;0})$ in
$T(\A D^{m,p}_{0,N;0})\equiv\A D^{m,p}_{0,N;0}\times \A^{m,p}_{0,N;0}$.
Before considering the general case, note that for any 
$\psi\in\accentset{\,\,\,\circ}\A D^{m,p}_{0,N;0}$ and for any 
$(\varphi,v)\in T(\accentset{\,\,\,\circ}\A D^{m,p}_{0,N;0})$,
${\mathcal E}(R_\psi\varphi,R_\psi v)={\mathsf R}_\psi\big({\mathcal E}(\varphi,v)\big)$,
where ${\mathsf R}_\psi\big(v,w\big):=\big(R_\psi v,R_\psi w\big)$. 
Now, take an arbitrary  $(\varphi_0,v_0)\in T(\accentset{\,\,\,\circ}\A D^{m,p}_{0,N;0})$. 
In view of $(b)$, we have that $v_0=R_{\varphi_0}u_0$ for some $u_0\in\accentset{\,\,\,\circ}{\A}^{m,p}_{0,N;0}$. 
In view of the invariance of $\mathcal E$ discussed above,
${\mathcal E}(\varphi_0,v_0)={\mathsf R}_{\varphi_0}\big({\mathcal E}(\id,u_0)\big)$.
The vector ${\mathsf R}_{\varphi_0}\big({\mathcal E}(\id,u_0)\big)$ is tangent to 
$T(\accentset{\,\,\,\circ}{\A}D^{m,p}_{0,N;0})$ as ${\mathcal E}(\id,u_0)$ is tangent 
to $T(\accentset{\,\,\,\circ}{\A}D^{m,p}_{0,N;0})$ and
${\mathsf R}_{\varphi_0}: (\psi,v)\mapsto (R_{\varphi_0}\psi,R_{\varphi_0} v)$,
$T(\accentset{\,\,\,\circ}\A D^{m,p}_{0,N;0})\to T(\accentset{\,\,\,\circ}\A D^{m,p}_{0,N;0})$,
is a diffeomorphism.
\finishproof

%%%%%%%%%%%%%%%%%%%%%%%%%%%%%%%%%%
\appendix
\section{Properties of asymptotic spaces}\label{sec:appendix_properties}
In this Appendix for the convenience of the reader we collect several results on the asymptotic spaces
that are used in the main body of the paper. Take $1<p<\infty$.

\medskip

\noindent{\em Asymptotic spaces:}
First, we discuss the remainder space $W^{m,p}_\delta(\R^d)$, $\delta\in\R$ (see \eqref{eq:remainder_space}). 

\begin{Lem}\label{lem:remainder_space}
Let $m>k+\frac{d}{p}$, $k\ge 0$. Then $W^{m,p}_{\delta}(\R^d)\subseteq C^k(\R^d)$  and
there exists $C>0$ such that for any $f\in W^{m,p}_{\delta}(\R^d)$, and for any multi-index $\alpha$ such that 
$0\le |\alpha|\le k$,
\[
\sup_{x\in\R^d}\big(\x^{\delta+\frac{d}{p}+|\alpha|}|\partial^\alpha f(x)|\big)\le C
\|f\|_{W^{m,p}_{\delta}}.
\]
Moreover, for any $\alpha$ such that $0\le|\alpha|\le k$,
\[
\x^{\delta+\frac{d}{p}+|\alpha|}|\partial^\alpha f(x)|=o(1)\,\,\,\text{as}\,\,\,|x|\to\infty\,.
\]
\end{Lem}

\noindent If $m>\frac{d}{p}$ and $\delta\ge-\frac{d}{p}$ then $W^{m,p}_\delta(\R^d)$ is a {\em Banach algebra}. 
We refer to \cite[Section 1]{McOwenTopalov2} for additional properties of the remainder space.

\medskip

Take $0\le n\le N$, $m>\frac{d}{p}$, $\ell\ge -n$, and consider the asymptotic space
$\A^{m,p}_{n,N;\ell}(\R^d)$. One has

\begin{Lem}\label{lem:properties_log_spaces}
\begin{itemize}
\item[$(a)$] If $n\le n_1$, $N\le N_1$, $0\le n_1\le N_1$ and $-n\le \ell_1\le \ell$ then
$\A^{m,p}_{n_1,N_1;\ell_1}(\R^d)\subseteq\A^{m,p}_{n,N;\ell}(\R^d)$ and the inclusion is bounded.
Moreover, we have $\A^{m,p}_{n+1,N+1;\ell}\subseteq W^{m,p}_{\gamma_n}$ and the inclusion is bounded;

\item[$(b)$] For any $j\ge 1$ and $k\ge 0$, multiplication by $\chi\frac{(\log r)^j}{r^k}$ is a bounded 
operator $\A^{m,p}_{n,N;\ell}(\R^d)\to\A^{m,p}_{n+k,N^-+k;\ell-k+j}(\R^d)$ for any $N^-<N$. Multiplication by
$\chi(r)/r^k$  is a bounded  operator $\A^{m,p}_{n,N;\ell}(\R^d)\to\A^{m,p}_{n+k,N+k;\ell-k}(\R^d)$;\footnote{By
definition the remainder space in $\A^{m,p}_{n,N+k;\ell}$ is $W^{m,p}_{\gamma_N+k}=W^{m,p}_{\gamma_{N+k}}$.}
\item[$(c)$] If $m>1+\frac{d}{p}$ then $u\mapsto\partial u/\partial x_j$,
$\A^{m,p}_{n,N;\ell}(\R^d)\to\A^{m-1,p}_{n+1,N+1;\ell-1}(\R^d)$ is bounded;
\item[$(d)$] For any $m_1, m_2>d/p$, $0\le n_1\le N_1$,  $0\le n_2\le N_2$, and $-n_1\le \ell_1$, $-n_2\le \ell_2$,
the point-wise multiplication $(u,v)\mapsto u\cdot v$,
\[
\A^{m_1,p}_{n_1,N_1;\ell_1}(\R^d)\times\A^{m_2,p}_{n_2,N_2;\ell_2}(\R^d)\to
\A^{\min\{m_1,m_2\},p}_{n_1+n_2,\min\{{N_1}^-+n_2,{N_2}^-+n_1\};\ell_1+\ell_2}(\R^d)
\]
is bounded for any ${N_1}^-<N_1$ and ${N_2}^-<N_2$.
\end{itemize}

\medskip

\noindent Item $(d)$ has the following two refinements:
\begin{itemize}
\item[$(d')$] For any $m_1, m_2>d/p$, $0\le n_1\le N_1$,  $0\le n_2\le N_2$,
point-wise multiplication $(u,v)\mapsto u\cdot v$,
\[
\A^{m_1,p}_{n_1,N_1;-n_1}(\R^d)\times\A^{m_2,p}_{n_2,N_2;-n_2}(\R^d)\to
\A^{\min\{m_1,m_2\},p}_{n_1+n_2,\min\{N_1+n_2,N_2+n_1\};-n_1-n_2}(\R^d)
\]
is bounded;
\item[$(d'')$] For any $m>d/p$, $1\le n\le N$,  $-n\le\ell_1,\ell_2$, point-wise multiplication 
$(u,v)\mapsto u\cdot v$,
\[
\A^{m,p}_{n,N;\ell_1}(\R^d)\times\A^{m,p}_{n,N;\ell_2}(\R^d)\to
\A^{m,p}_{n,N;\ell_1+\ell_2}(\R^d)\,,
\]
is bounded. 
\end{itemize}
In particular, the asymptotic space $\A^{m,p}_{n,N;\ell}(\R^d)$ with $-n\le \ell\le 0$ is a Banach algebra
with respect to the point-wise multiplication of functions.
\end{Lem}

\noindent We refer to \cite[Appendix B]{McOwenTopalov2} for the proof of this lemma.

\medskip

These asymptotic spaces are useful in studying the Laplace operator and its inverse. However,
due to the logarithmic potential when $d=2$, we need to introduce some additional spaces.
Denote by $\A^{m+1,p}_{1^*,N}(\R^2)$ the space of $u\in H^{m+1}_{loc}(\R^2)$ such that,
\[
u(x)=\chi(r)\Big(a_0^*\log r +\frac{a_1(\theta)}{r}+\cdots+\frac{a_N(\theta)}{r^N}\Big)+f(x)\,,
\]
where $a_0^*=\text{\rm const}$, $a_k\in H^{m+1+N-k,p}(\s^1)$, $k=1,...,N$, and 
$f\in W^{m,p}_{\gamma_N}(\R^d)$. The space $\A^{m+1,p}_{1^*,N}(\R^2)$ is equipped with a
norm similar to \eqref{eq:the_norm}. Clearly, the space $\A^{m+1,p}_{1^*,N}(\R^2)$ is closed subspace
in $\A^{m+1,p}_{0,N;1}(\R^2)$ and the embedding 
$\A^{m+1,p}_{1^*,N}(\R^2)\subseteq\A^{m+1,p}_{0,N;1}(\R^2)$ is continuous.

\begin{Lem}\label{lem:inverse_laplacian_remainder_space}
Suppose that $d\ge 2$ and $m>1+\frac{d}{p}$.
\begin{itemize}
\item[$(a)$]  For $d=2$, there exists a bounded operator,
\[
K  :  W^{m-1,p}_{\gamma_N+2}(\R^2)\to\A^{m+1,p}_{1^*,N}(\R^2)\,,
\]
such that $\Delta Kg=g$. More specifically, $u=Kg$ is of the form
\begin{equation}\label{eq:d=2}
u(x)=\chi(r)\Big(a_0^*\log r+\frac{a_1(\theta)}{r}+...+\cdots\frac{a_N(\theta)}{r^N}\Big)+f(x)\,,
\,\,\,\,f\in W^{m+1,p}_{\gamma_N},
\end{equation}
where $a_0^*$ is a constant and $a_k(\theta)/r^k$, $k=1,...,N$ are harmonic functions on 
$\R^2\setminus\{0\}$;
\item[$(b)$] For $d\ge 3$, there exists a bounded operator,
\[
K  :  W^{m-1,p}_{\gamma_N+2}(\R^d)\to\A^{m+1,p}_{d-2,N}(\R^d)\,,
\]
such that $\Delta Kg=g$. More specifically, $u=Kg$ is of the form
\begin{equation}\label{eq:d>=3}
u(x)=\chi(r)\Big(\frac{a_{d-2}}{r^{d-2}}+\frac{a_{d-1}(\theta)}{r^{d-1}}+\cdots+\frac{a_N(\theta)}{r^N}\Big)+f(x),
\,\,\,\,f\in W^{m+1,p}_{\gamma_N},
\end{equation}
where $a_{d-2}=\text{\rm const}$ and $a_k(\theta)/r^k$, $k=d-2,...,N$ are harmonic functions on 
$\R^d\setminus\{0\}$.
\end{itemize}
\end{Lem}

\noindent For the proof of this Lemma we refer to \cite[Section 3]{McOwenTopalov2}.

\begin{Rem}
Using basic facts from the theory of distributions one can easily see that the asymptotic terms in 
\eqref{eq:d=2} and  \eqref{eq:d>=3} are linear combinations with 
constant coefficients of the derivatives $\partial^\alpha{\cal E}_d$, $|\alpha|\le N$, where
${\cal E}_2:=\frac{1}{2\pi}\log r$ and ${\cal E}_d=c_{d-2}/r^{d-2}$, $d\ge 3$, is the fundamental solutions of the 
Laplace operator in $\R^d$. In particular, the coefficients $a_k(\theta)$ appearing in \eqref{eq:d=2} and 
\eqref{eq:d>=3} are  polynomial functions in $\theta\equiv\frac{x}{|x|}$.
\end{Rem}

\begin{Rem}\label{rem:leading_term}
In fact, it follows from the proof of Lemma 3.1 in \cite{McOwenTopalov2} that the coefficients $a_0^*$ and
$a_{d-2}$ of the leading asymptotics of $K g$, $g\in W^{m-1,p}_{\gamma_N+1}$, are
equal to $\int_{\R^d}g(x)\,{\rm d}x$. More generally, the proof of Lemma 3.1 in \cite{McOwenTopalov2} shows that
for any $1\le k\le N$ the coefficient $a_k(\theta)$ in \eqref{eq:d=2} vanishes if and only if for any 
homogeneous harmonic polynomial $h$ of degree $k$, $\iint_{\R^2} g h\,{\rm d}x=0$.
A similar statement also holds for the coefficients in \eqref{eq:d>=3}
\end{Rem}

\medskip

\noindent {\em Additional spaces and results:}
Take $R>0$ and consider the open set $B^c_R=\big\{x\in\R^d\,\big|\,|x|>R\big\}\subseteq\R^d$.
For technical reasons we will also need the following modification of the asymptotic space $\A^{m,p}_{n,N;\ell}(\R^d)$: 
For any $m>\frac{d}{p}$, $1\le n\le N$, $-n\le \ell$, consider the space $\A^{m,p}_{n,N;\ell}(B^c_R)$ of
functions $u\in H^{m,p}_{loc}(B^c_R)$ such that,
\[
u(x)=\Big(\frac{a_n^0+\cdots a_n^{n+\ell}(\log r)^{n+\ell}}{r^n}+\cdots+
\frac{a_N^0+\cdots a_N^{N+\ell}(\log r)^{N+\ell}}{r^N}\Big)+f(x)\,,
\]
where $a_k^j\in H^{m+1+N-k}(\s^{d-1})$, $0\le j\le k+\ell$ and $n\le k\le N$, and 
$f\in W^{m,p}_{\gamma_N}(B^c_R)$. The space $f\in W^{m,p}_{\gamma_N}(B^c_R)$ is defined in the same
way as $W^{m,p}_{\gamma_N}(\R^d)$ with the only difference that $\R^d$ is replaced by $B^c_R$.
The space  $\A^{m,p}_{n,N;\ell}(B^c_R)$ is equipped with the norm,
\begin{equation}\label{eq:the_norm_complement}
\|u\|_{\A^{m,p}_{n,N;\ell}(B^c_R)}:=
\sum_{n\le k\le N, 0\le j\le k+\ell}\frac{(\log R)^j}{R^k}\|a_k^j\|_{ H^{m+1+N-k,p}(\s^{d-1})}+
\|f\|_{W^{m,p}_{\gamma_N}(B^c_R)}\,.
\end{equation}
Note that for $u\in\A^{m,p}_{n,N;\ell}(B^c_{R_0})$, $\|u\|_{\A^{m,p}_{n,N;\ell}(B^c_R)}\to 0$ as $R\to\infty$.
It is also clear that the restriction $u\mapsto u|_{B^c_r}$, $\A^{m,p}_{n,N;\ell}(\R^d)\to\A^{m,p}_{n,N;\ell}(B^c_R)$
is bounded.

\medskip

Denote $B_R:=\big\{x\in\R^d\,\big|\,|x|<R\big\}\subseteq\R^d$. Items $(a)$ and $(c)$ of the following 
Lemma are proved in the same way as Lemma 6.9 in \cite{McOwenTopalov1}. The proof of item $(b)$ follows the lines of 
the proof of Lemma \ref{lem:properties_log_spaces}.

\begin{Lem}\label{lem:asymptotic_spaces_complement}
Assume that $m>\frac{d}{p}$, $0\le n\le N$, $-n\le \ell$, $R>0$. Then we have:
\begin{itemize}
\item[$(a)$] If $u\in\A^{m,p}_{n,N;\ell}(B^c_R)$ for some $R>0$ then there exists
${\tilde u}\in\A^{m,p}_{n,N;\ell}(\R^d)$ such that ${\tilde u}|_{B^c_R}=u$.
Moreover, $u\in\A^{m,p}_{n,N;\ell}(\R^d)$ if and only if the restriction $u|_{B^c_R}\in
\A^{m,p}_{n,N;\ell}(B^c_R)$ and 
$u|_{B_{R'}}\in H^{m,p}(B_{R'})$ for some $R'>R$. The norm $\|u\|_{\A^{m,p}_{n,N;\ell}(\R^d)}$ is 
equivalent to the norm $\|u\|_{\A^{m,p}_{n,N;\ell}(B^c_R)}+\|u\|_{H^{m,p}(B_{R'})}$;
\item[$(b)$] The statement of Lemma \ref{lem:properties_log_spaces} holds with $\R^d$ replaced by $B^c_R$ in 
the notation of the asymptotic spaces;
\item[$(c)$] Assume in addition that $-n\le\ell\le 0$. 
Then there exist $C>0$ and $R_0>0$ such that for any $R\ge R_0$ for any
$u\in W^{m,p}_{\gamma_N}(B_R^c)$ and for any $v\in\A^{m,p}_{n,N;\ell}(B^c_R)$ we have 
$u\cdot v\in W^{m,p}_{\gamma_N}(B_R^c)\subseteq\A^{m,p}_{n,N;\ell}(B^c_R)$ and
\[
\|u\cdot v\|_{\A^{m,p}_{n,N;\ell}(B^c_R)}\le C \|u\|_{\A^{m,p}_{n,N;\ell}(B^c_R)}
\|v\|_{\A^{m,p}_{n,N;\ell}(B^c_R)}.
\]
\end{itemize}
\end{Lem}

\noindent  Recall that $\Id$ the identity $d\times d$-matrix and
$[{\rm d}_x\varphi]=\big(\frac{\partial\varphi^k}{\partial x_l}(x)\big)_{1\le k,l\le d}$. 
Lemma  \ref{lem:asymptotic_spaces_complement} and the arguments used in the proof of Lemma 6.10 
in \cite{McOwenTopalov1} imply

\begin{Lem}\label{lem:inverse_jacobian}
The mapping,
\[
\varphi\mapsto [{\rm d}\varphi]^{-1}-\Id,\,\,\,\A D^{m,p}_{0,N;0}\to\A^{m-1,p}_{1,N+1;-1}\,,
\]
is real-analytic.
\end{Lem}

\medskip

\noindent The following Remarks clarify some choices made in the definition of the asymptotic spaces.

\begin{Rem}\label{rem:the_cut-off}
If we replace the cut-off function $\chi(r)$ in the definition of the asymptotic space
$\A^{m,p}_{n,N;\ell}$ by any other cut-off function $\chi\in C^\infty(\R)$, $\chi(r)=1$ for $r>R_2$, and 
$\chi(r)=0$ for $0\le r\le R_1$ where $0<R_1<R_2$, we will get the same asymptotic space equipped with 
an equivalent norm (see Remark 2.2 in \cite{McOwenTopalov2}).
\end{Rem}

\begin{Rem}\label{rem:the_sphere}
Take $0<R_1<R_2$ and consider the annulus $B_{R_1,R_2}:=\big\{x\in\R^d\,\big|\,R_1<|x|<R_2\big\}\subseteq\R^d$.
Let $a\in H^{m,p}(\s^{d-1})$, $m>d/p$, and let us denote by ${\hat a}$ the homogeneous function of degree zero
in $\R^d\setminus\{0\}$, ${\hat a}(x):=a(\theta)$, $\theta=\frac{x}{|x|}$. 
Then one can easily confirm that ${\hat a}\in H^{m,p}(B_{R_1,R_2})$ and that the map
$a\mapsto{\hat a}$, $H^{m,p}(\s^{d-1})\to H^{m,p}(B_{R_1,R_2})$ is bounded.
\end{Rem}

%%%%%%%%%%%%%%%%%%%%%%%%%%%%%%%%%%
\section{Examples}\label{sec:examples}
In this Appendix we consider two Examples. In the first one we 
construct a divergence free asymptotic vector field without log terms so that the
corresponding solution of the Euler equation develops log terms.
In the second example we construct a divergence free vector field with
compact support that develops a non-vanishing asymptotic term $\chi(r)\,a_{d+1}(\theta)/r^{d+1}$.

\medskip

\noindent{\em Example 1.}
Here we construct a divergence free vector field 
$u_0\in\A^{m,p}_3$ such that the solution 
$u\in C^0\big([0,T],\accentset{\,\,\,\circ}\A^{m,p}_{3;0}\big)\cap 
C^1\big([0,T],\accentset{\,\,\,\circ}\A^{m-1,p}_{3;0}\big)$ 
given by Theorem \ref{th:main} has log terms, i.e.\ there exists $0<\tau\le T$ such that
\begin{equation*}
u(\tau)\in\accentset{\,\,\,\circ}\A^{m,p}_{3;0}\setminus\A^{m,p}_3\,.
\end{equation*}
This shows that $\A^{m,p}_N$ with $N\ge 3$ is {\em not} invariant under the Euler flow, and establishes the
necessity of using the asymptotic spaces with log terms.
For simplicity, we restrict our attention to the case when $d=2$. In this case we can identify the complex plane 
$\C$ with $\R^2$ and use complex notations. In particular, we set $z=x+ i y=e^{i\phi}$, $\phi\in\R/2\pi\Z$, and
note that $\Delta=4\partial_z\partial_{\bar z}$ where $\partial_z=(\partial_x-i\partial_y)/2$ and
$\partial_{\bar z}=(\partial_x+i\partial_y)/2$.
Let $H$ be a real valued $C^\infty$-smooth Hamiltonian function on $\R^2=\{(x,y)\}$ and let $u_0=(-H_y, H_x)$ be 
the corresponding Hamiltonian vector field. Then, by a straightforward  computation, we have
\begin{equation}\label{eq:Q(d=2)}
Q(u_0)=\tr\big([{\rm d}u_0]^2\big)=-2\det\left[
\begin{array}{cc}
H_{xx}&H_{xy}\\
H_{yx}&H_{yy}
\end{array}
\right]
=8\det\left[
\begin{array}{cc}
H_{zz}&H_{z{\bar z}}\\
H_{{\bar z}z}&H_{{\bar z}{\bar z}}
\end{array}
\right]\,.
\end{equation}
Now, consider the Hamiltonian
\[
H:=\left[\left(\frac{z}{\bar z}+\frac{\bar z}{z}\right)+
\alpha\left(\Big(\frac{z}{\bar z}\Big)^2+\Big(\frac{\bar z}{z}\Big)^2\right)\right]\chi(r),
\]
where $\alpha\ne 0$ is a given constant.
Then we have
$H_z=\left[-\frac{\bar z}{z^2}+\frac{1}{\bar z}+
2\alpha\Big(\frac{z}{{\bar z}^2}-\frac{{\bar z}^2}{z^3} \Big)\right]\chi(r)+C^\infty_{comp}$
where $C^\infty_{comp}$ stands for a term in $C^\infty$ with compact support.
This implies that
\[
H_{z{\bar z}}=\left[-\frac{1}{{\bar z}^2}-\frac{1}{z^2}-
4\alpha\Big(\frac{\bar z}{z^3}+\frac{z}{{\bar z}^3}\Big)\right]\chi(r)+C^\infty_{comp},\quad
H_{zz}=\left[2\frac{\bar z}{z^3}+6\alpha\frac{{\bar z}^2}{z^4}+2\alpha\frac{1}{{\bar z}^2}\right]\chi(r)+
C^\infty_{comp}\,.
\]
As $H$ is real valued we have $H_{{\bar z}{\bar z}}=\overline{(H_{zz})}$.
This and \eqref{eq:Q(d=2)} then imply that
\[
Q(u_0)=32\alpha\left(\frac{1}{z{\bar z}^3}+\frac{1}{{\bar z}z^3}\right)\chi(r)+\cdots=
32\alpha\,\frac{e^{2i\phi}+e^{-2i\phi}}{r^4}\chi(r)+\cdots ,
\]
where $...$ stands for a term in $C^\infty$ with compact support plus an asymptotic term $\frac{a(\phi)}{r^4}\chi(r)$ 
such that the Fourier series of $a(\phi)$ does not contain exponents $e^{\pm 2i\phi}$. Then, in view of
Lemma \ref{lem:inverse_laplacian_remainder_space} $(a)$, formula \eqref{eq:Delta(l=1)} and \eqref{eq:Delta(l=0)},
\[
\Delta^{-1} \circ Q(u_0)=-4\alpha\log(z{\bar z})\,\Big(\frac{1}{z^2}+\frac{1}{{\bar z}^2}\Big)\chi(r)+\cdots,
\] 
where $...$ stands for an asymptotic term of the form
$\big(a_0^*\log r+\frac{a_1(\phi)}{r}+\frac{a_2(\phi)}{r^2}\big)\chi(r)$, where $a_0^*=\const$ and
$a_1(\phi)$ and $a_2(\phi)$ are trigonometric polynomials of $\phi\in\R/2\pi\Z$, and a reminder term in 
$W^{m+1,p}_{\gamma_2}$.
This shows that $\nabla\circ\Delta^{-1} \circ Q(u_0)$ has a non-vanishing asymptotic term 
$\frac{b(\phi)\log r}{r^3}\chi(r)$ where $b(\phi)$ is an $\R^2$-valued trigonometric polynomial of $\phi$. 
This together with Theorem \ref{th:main} and Lemma \ref{lem:no_pressure_term} 
then shows that, for $0<\tau\le T$ small enough, $u(\tau)\in\accentset{\,\,\,\circ}\A^{m,p}_{3;0}\setminus\A^{m,p}_3$.
Hence, the space $\A^{m,p}_3$ is {\em not} invariant with respect to the Euler flow.
Finally, note that the constructed divergence free vector field $u_0$ is of the form
\[
u_0=\frac{c(\phi)}{r}\chi(r)+C^\infty_{comp}
\] 
where $c(\phi)$ is an $\R^2$-valued trigonometric polynomial of $\phi\in\R/2\pi\Z$.

\medskip\medskip

\noindent{\em Example 2.}
Below, for the convenience of the reader, we construct a divergence free $C^\infty$-smooth vector field with 
compact support so that the solution of the Euler equation given by Theorem \ref{th:main} has a non-vanishing 
asymptotic term (cf. \cite{DobrShaf} for a different argument).
As above, we restrict our attention to the case $d=2$, identify $\R^2$ with the complex plane $\C$, and
use complex notations. 
Let $a :\R\to\R$ be a $C^\infty$-smooth function so that $a(\varrho)=1$ for $|\varrho|\le 1$ and
$a(\varrho)=0$ for $|\varrho|\ge 2$. Consider the Hamiltonian
\[
H:=\big(z+{\bar z}\big) a
\]
where $a$ stands for $a(z{\bar z})$ and let
$u_0=\big(-H_y,H_x\big)$ be the Hamiltonian vector field of $H$. We have
\[
H_{\bar z}= z \big(z+{\bar z}\big) a'+a
\]
where $a'$ stands for $a'(z{\bar z})$.
Continuing this computation, we have
\[
H_{z{\bar z}}=\Big((z{\bar z})a''+2 a'\Big)\big(z+{\bar z}\big)
\]
\[
H_{{\bar z}{\bar z}}=\Big((z{\bar z})a''+2 a'\Big)z+z^3a''
\]
and, as $H_{zz}=\overline{(H_{{\bar z}{\bar z}})}$,

\[
H_{zz}=\Big((z{\bar z})a''+2 a'\Big){\bar z}+{\bar z}^3a''.
\]
Then, by \eqref{eq:Q(d=2)}, we have
\begin{eqnarray*}
Q(u_0)&=&8 \big(H_{zz} H_{{\bar z}{\bar z}}-(H_{z{\bar z}})^2\big)\\
&=&-8\Big(2(z{\bar z})a'a''+4(a')^2\Big)\big(z^2+{\bar z}^2\big)+\cdots,
\end{eqnarray*}
where $\cdots$ stands for a term that dependents only on $|z|$.
Further, consider the integral
\[
\iint\limits_{\R^2} Q(u_0)\,{\bar z}^2\,{\rm d}x\,{\rm d}y=-8\pi\int_0^\infty\!\!\!\!A(\varrho)\,{\rm d}\varrho
\]
where $A(\varrho):=2a'(\varrho)a''(\varrho)\varrho^3+4(a'(\varrho))^2\varrho^2$ and we passed to the new variable
$\varrho=r^2$ in the integral. As $A(\varrho)=\big((a'(\varrho))^2\varrho^3\big)'+(a'(\varrho))^2\varrho^2$ we
obtain that $\int_0^\infty A(\varrho)\,d\varrho=\int_0^\infty(a'(\varrho))^2\varrho^2\,{\rm d}\varrho$, and
hence
\begin{equation}
\iint\limits_{\R^2} Q(u_0)\,(x^2-y^2)\,{\rm d}x\,{\rm d}y<0\,.
\end{equation}
In view of Lemma \ref{lem:vanishing_lemma} below we have
$\iint_{\R^2}Q(u_0)\,{\rm d}x\,{\rm d}y=0$.
This together with Lemma \ref{lem:inverse_laplacian_remainder_space} and Remark \ref{rem:leading_term} then 
implies that $\Delta^{-1}\circ Q(u_0)$ belongs to the asymptotic space $\A^{m+1,p}_{2,N;0}$ for any given $N\ge 3$ and
\[
\Delta^{-1}\circ Q(u_0)=\frac{a_2(\theta)}{r^2}\chi(r)+\cdots
\]
with $a_2\ne 0$. Theorem \ref{th:main} and Lemma \ref{lem:no_pressure_term} 
then show that the solution
$u\in C^0\big([0,T],\accentset{\,\,\,\circ}\A^{m,p}_{N;0}\big)\cap 
C^1\big([0,T],\accentset{\,\,\,\circ}\A^{m-1,p}_{N;0}\big)$
of the Euler equation with initial data $u_0$ constructed above
has a non-vanishing asymptotic term $\chi(r) a_3(\theta)/r^3$. 

\begin{Rem}
In view of Lemma \ref{lem:vanishing_lemma} below, Lemma \ref{lem:inverse_laplacian_remainder_space} and 
Remark \ref{rem:leading_term}, there is no $C^\infty$-smooth divergence free 
vector field $u_0$ with compact support in $\R^2$ so that $\Delta^{-1}\circ Q(u_0)$ has non-vanishing asymptotic terms 
$\chi(r)\,a_0^*\log r$ or $\chi(r)\,a_1(\theta)/r$. A similar statement also holds for $d\ge 3$.
\end{Rem}

\begin{Rem}\label{rem:general_position}
Denote by $\accentset{\circ}{\mathcal S}$ the linear space of divergence free vector fields in $\R^2$ whose
components belong to the Schwartz class of functions ${\mathcal S}(\R^2)$.  
Example 2 shows that the bounded quadratic form $u\mapsto\iint_{\R^2} Q(u)\,(x^2-y^2)\,{\rm d}x\,{\rm d}y$, 
$\accentset{\circ}{\mathcal S}\to\R$, 
does not identically vanish. As this is a real analytic function on $\accentset{\circ}{\mathcal S}$ we see that
there is an open and dense set of vector fields $u_0$ in $\accentset{\circ}{\mathcal S}$ so that for any 
$N\ge 3$ the solution 
$u\in C^0\big([0,T],\accentset{\,\,\,\circ}\A^{m,p}_{N;0}\big)\cap 
C^1\big([0,T],\accentset{\,\,\,\circ}\A^{m-1,p}_{N;0}\big)$
of the Euler equation with initial data $u_0$ has a non-trivial asymptotic term $\chi(r) a_3(\theta)/r^3$.
Moreover, one can conclude from Remark \ref{rem:analyticity_in_t} that for almost all $t\in[0,T]$ the asymptotic
part of the solution $u(t)$ is non-trivial.
Similar statements also hold for $d\ge 3$.
\end{Rem}

\begin{Lem}\label{lem:vanishing_lemma}
Let $u_0$ be a $C^\infty$-smooth divergence free vector field with compact support in $\R^d$, $d\ge 2$.
Then 
\[
\int_{\R^d}Q(u_0)\,{\rm d}x_1\ldots{\rm d}x_d=0\quad\text{and}\quad
\int_{\R^d}x_k\,Q(u_0)\,{\rm d}x_1\ldots{\rm d}x_d=0
\]
for any $1\le k\le d$.
\end{Lem}

\begin{proof}
The Lemma follows from \eqref{eq:the_trick2} and the Stokes' theorem.
In fact, let $u_0$ be a $C^\infty$-smooth divergence free vector field with compact support in $\R^d$, $d\ge 2$.
By \eqref{eq:the_trick2} 
\[
Q(u_0)=\Div X\quad\text{where}\quad X:=u_0\cdot\nabla u_0.
\]
As the vector field $X$ has compact support we see from the Stokes' theorem that
\[
\int_{\R^d}Q(u_0)\,{\rm d}x_1\ldots{\rm d}x_d=\int_{\R^d} (\Div X)\,{\rm d}x_1\ldots{\rm d}x_d=0.
\] 
In order to prove the second equality note that for any $1\le k\le d$ we have
\begin{equation}\label{eq:lie_derivative}
L_X\big(x_k\,{\rm d}x_1\wedge\ldots\wedge{\rm d}x_d\big)=
X_k\,{\rm d}x_1\wedge\ldots\wedge{\rm d}x_d+x_k (\Div X)\,{\rm d}x_1\wedge\ldots\wedge{\rm d}x_d
\end{equation}
where $L_X$ denotes the Lie derivative in the direction of the vector field $X$ and $X_k$ is the $k$-th
component of $X$. As
$L_X\big(x_k\,{\rm d}x_1\wedge\ldots\wedge{\rm d}x_d\big)=
{\rm d}\big(i_X(x_k{\rm d}x_1\wedge\ldots\wedge{\rm d}x_d)\big)$
and as $X$ has compact support we get from \eqref{eq:lie_derivative} and the Stokes' theorem that
\begin{equation}\label{eq:x_k}
\int_{\R^d}x_k(\Div X)\,{\rm d}x_1\ldots{\rm d}x_d=
-\int_{\R^d}X_k\,{\rm d}x_1\ldots{\rm d}x_d\,.
\end{equation}
Now, using that $X=u_0\cdot\nabla u_0$ and the fact that $u_0$ preserves the volume form 
${\rm d}x_1\wedge\ldots\wedge{\rm d}x_d$ we conclude that
$X_k\,{\rm d}x_1\wedge\ldots\wedge{\rm d}x_d=L_{u_0}\big(u_{0k}\,{\rm d}x_1\ldots{\rm d}x_d\big)=
{\rm d}\big(i_{u_0}(u_{0k}\,{\rm d}x_1\ldots{\rm d}x_d)\big)$ where $u_{0k}$ is the $k$-th component
of the vector field $u_0$. As $u_0$ has compact support we again conclude from
the Stokes' theorem that
\[
\int_{\R^d}X_k\,{\rm d}x_1\ldots{\rm d}x_d=0\,.
\]
This together with \eqref{eq:x_k} then completes the proof of the Lemma.
\end{proof}

%%%%%%%%%%%%%%%%%%%%%%%%%%%%%%%%%


\begin{thebibliography}{99}   

\bibitem{Arn} V. Arnold, {\em Sur la geometri{\'e} differentielle des groupes de Lie
de dimension infinie et ses applications {\`a} l'hydrodynamique des fluids parfaits},
Ann. Inst. Fourier, $\bf 16$(1966), no. 1, 319-361

%\bibitem{AK} V. Arnold, B. Khesin, {\em Topological methods in hydrodynamics},
%Springer, 1998

\bibitem{BardosTiti}  K. Bardos, E. Titi, {\em Euler equations for an ideal incompressible fluid},
Russian Math. Surveys, $\bf 62$(2007), no. 3, 409-451

\bibitem{BS1} I. Bondareva, M. Shubin, {\em Growing asymptotic solutions of
the Korteweg-de Vries equation and of its higher analogues}, Dokl. Akad. Nauk SSSR,
$\bf 267$(1982), no. 5, 1035-1038

\bibitem{BS2} I. Bondareva, M. Shubin, {\em Uniqueness of the solution of the Cauchy problem
for the Korteweg-de Vries equation in classes of increasing functions},
Vestnik Moskov. Univ. Ser. I Mat. Mekh, $\bf 1985$, no. 3, 35-38

\bibitem{BS3} I. Bondareva, M. Shubin, {\em Equations of Korteweg-de Vries type in classes of
increasing functions}, J. Soviet Math., $\bf 51$(1990), no. 3, 2323-2332

%\bibitem{BKShTe} A. Boutet de Monvel, A. Kostenko, D. Shepelsky, G. Teschl,
%{\em Long-time asymptotics for the Camassa-Holm equation}, SIAM J. Math Anal.,
%$\bf 41$(2009), no. 4, 1559-1588

\bibitem{BB} J. P. Bourguignon, H. Brezis, {\em Remarks on the Euler equation},
J. Func. Anal., $\bf 15$(1974), 341-363

%\bibitem{Bur} J. Burgers, {\em A mathematical model illustrating the theory of turbulence},
%Adv. Appl. Mech., $\bf 1$(1948), 171-199

\bibitem{Bran} L. Brandolese, {\em Space-time decay of Navier-Stokes flows invariant under rotations},
Math. Ann., $\bf 329$(2004), no. 4, 685-706

\bibitem{BranMe} L. Brandolese, Y. Meyer, {\em On the instantaneous spreading for the Navier-Stokes system in 
the whole space. A tribute to J. L. Lions.},  ESAIM Control Optim. Calc. Var., $\bf 8$(2002), 273-285

\bibitem{C} M. Cantor, \textit{Perfect fluid flows over $\R^n$ with asymptotic conditions}, J. Func. Anal., 
$\bf 18$(1975), 73-84

\bibitem{CH} R. Camassa, D. Holm, {\em An integrable shallow water equation with
peaked solitons}, Phys. Rev. Lett, $\bf 71$(1993), 1661-1664

\bibitem{Con2} A. Constantin, {\em Existence of permanent and breaking waves for a
shallow water equation: a geometric approach}, Ann. Inst. Fourier, Grenoble, $\bf 50$(2000), no. 2,
321-362

%\bibitem{Con3} A. Constantin, {\em On the scattering problem for Camassa-Holm equation},
%Proc. R. Soc. Lond. A, $\bf 457$(2001), 953-970

%\bibitem{Con4} A. Constantin, {\em The trajectories of particles in Stokes waves}, Invent. math., 
%$\bf 166$(2006), 523-535

%\bibitem{CE2}  A. Constantin, J. Escher, {\em Global existence and blow-up for a shallow water
%equation}, Annali Sc. Norm. Sup. Pisa, $\bf 26$(1998), 303-328

%\bibitem{CE3} A. Constantin, J. Escher, {\em On the blow-up rate and
%the blow-up set of breaking waves for a shallow water equation}, Math. Z., 
%$233$(2000), 75-91

%\bibitem{CE4} A. Constantin, J. Escher,
%{\em Global weak solutions for a shallow water equation},
%Indiana Univ. Math. J., $\bf 47$(1998), no. 4, 1527-1545

%\bibitem{CLan} A.Constantin, D. Lannes, {\em The hydrodynamical relevance of the Cama\-ssa-Holm and
%Degasperis-Procesi equation}, Arch. Rational Mech. Anal., $\bf 192$(2009), 165-186

%\bibitem{CMcK} A. Constantin, H. McKean, {\em A shallow water equation on the circle}, 
%Comm. Pure Appl. Math., $\bf 52$(1999), 949-982

%\bibitem{CK2} A. Constantin, B. Kolev, {\em Geodesic flow on the diffeomorphism group
%of the circle}, Comment. Math. Helv., $\bf 78$(2003), 787-804

\bibitem{DKT} C. De Lellis, T. Kappeler, P. Topalov, {\em Low regularity solutions of
the Camassa-Holm equation}, Comm. Partial Differential Equations, $\bf 32$(2007), no. 1-3, 87-126

\bibitem{Dieudonne} J. Dieudonn\'e, Foundations of Modern Analysis, Academic Press, 1969

\bibitem{EM} D. Ebin, J. Marsden, {\em Groups of diffeomorphisms and the motion
of an incompressible fluid}, Ann. Math., $\bf 92$(1970), 102-163

\bibitem{DobrShaf} S. Dobrokhotov, A. Shafarevich, {\em Some integral identities and remarks on the decay
at infinity of the solutions of the Navier-Stokes Equations in the entire space},
Russian J. Math. Phys., $\bf 2$(1994), no. 1, 133-135

\bibitem{Erdelyi} A. Erd{\'e}lyi, Asymptotic Expansions, Dover Publications, Inc., 1956

%\bibitem{FF} A. Fokas, B. Fuchssteiner, {\em Symplectic structures,
%their B{\"a}cklund transformation and hereditary symmetries},
%Physica D, $\bf 4$(1981), 47-66 

%\bibitem{Ham} R. Hamilton, {\em The inverse function theorem of Nash and Moser},
%Bull. Amer. Math. Soc., $\bf 7$(1982), 66-222

%\bibitem{HS1} D. Holm, M. Staley, {\em Wave structure and nonlinear balances in a family of evolutionary
%PDEs}, SIAM J. Applied Dynamical Systems, $\bf 2$(2003), no. 3, 323-380

%\bibitem{HS2} D. Holm, M. Staley, {\em Nonlinear balance and exchange of stability in dynamics of solitons,
%peakons, ramps/cliffs and leftons in $1+1$ nonlinear PDE}, Phys. Lett. A, $\bf 308$(2003), 437-444

\bibitem{IKT} H. Inci, T. Kappeler, P. Topalov, {\em On the regularity of the composition of
diffeomorphisms}, Mem. Amer. Math. Soc., $\bf 226$(2013), no. 1062

\bibitem{KLT1} T. Kappeler, E. Loubet, P. Topalov, {\em Analyticity of Riemannian exponential maps on 
$\text{\rm Diff}(\mathbb{T})$}, J. Lie Theory,  $\bf 17$(2007), no. 3, 481-503

\bibitem{KLT2}  T. Kappeler, E. Loubet, P. Topalov, {\em  Riemannian exponential maps of the diffeomorphism 
groups of $\T^2$}, Asian J. Math., $\bf 12$(2008), no. 3, 391-420

\bibitem{KPST} T. Kappeler, P. Perry, M. Shubin, P. Topalov, {\em Solutions of mKdV in classes of 
functions unbounded at infinity}, J. Geom. Anal., $\bf 18$(2008), no. 2, 443-477

\bibitem{Kato1} T. Kato, {\em On classical solutions of the two-dimensional non-stationary Euler equation}
Arch. Rational Mech. Anal., $\bf 25$(1967), 188-200.

\bibitem{Kato2} T. Kato, {\em Quasi-linear equations of evolution, with applications to partial differential equations}
Lecture Notes in Math., $\bf 448$, Springer, Berlin, 1975

%\bibitem{Kato3} T. Kato, {\em The Cauchy problem for quasi-linear symmetric hyperbolic systems},
%Arch. Ration. Mech. Anal., $\bf 58$(1975), 181-205

%\bibitem{KM} B. Khesin, G. Misiolek, {\em Euler equations on homogeneous spaces and
%Virasoro orbits}, Adv. Math, $\bf 176$(2003), 116-144

\bibitem{KR} I. Kukavica, E. Reis, {\em Asymptotic expansion for solutions of the Navier-Stokes equations 
with potential forces}, J. Differential Equations, $\bf 250$(2011), no. 1, 607-622

\bibitem{Lang} S. Lang, Differential manifolds,
Addison-Wesley Series in Mathematics, 1972

%\bibitem{McK} H. McKean, {\em Fredholm determinants and Camassa-Holm hierarchy},
%Comm. Pure Appl. Math., $\bf 56$(2003), no. 5, 638-680

\bibitem{McOwen} R. McOwen, {\em The behavior of the Laplacian on weighted Sobolev spaces},
Comm. Pure Appl. Math., $\bf 32$(1979), no. 6, 783-795

\bibitem{McOwenTopalov1} R. McOwen, P. Topalov, {\em  Asymptotics in shallow water waves},
Discrete Contin. Dyn. Syst., $\bf 35$(2015), no. 7, 3103-3131

\bibitem{McOwenTopalov2} R. McOwen, P. Topalov, {\em Groups of asymptotic diffeomorphisms}, 
to appear in Discrete Contin. Dyn. Syst., arXiv:1503.04850

\bibitem{Menikoff}  A. Menikoff, {\em The existence of unbounded solutions of the Korteweg-de Vries equation},
Comm. Pure Appl. Math., $\bf 25$(1972), 407-432

%\bibitem{Milnor} J. Milnor, {\em Remarks on infinite-dimensional Lie groups},
%Les Houches, Session XL, 1983, Elsevier Science Publishers B.V., 1984

\bibitem{Mis1} G. Misiolek, {\em A shallow water equation as a geodesic flow
on the Bott-Virasoro group}, J. Geom. Phys., $\bf 24$(1998), 203-208

%\bibitem{Mis2} G. Misiolek, {\em Classical solutions of the periodic
%Camassa-Holm equation}, GAFA, $\bf 12$(2002), 1080-1104

%\bibitem{OK} V. Ovsienko, B. Khesin, {\em Korteweg-de Vries superequations
%as an Euler equation}, Functional Anal. Appl., $\bf 21$(1987), 81-82

%\bibitem{Tol} J. Toland, {\em Stokes waves}, Topological Methods in Nonlinear Analysis,
%$\bf 7$(1996), 1-48

\bibitem{Serfati} P. Serfati, {\em \'Equations d'Euler et holomorphies \`a faible regularit\'e
spatiale}, C. R. Acad. Sci. Paris, $\bf 320$(1994), no. 2, S\'erie I, 175-180

\bibitem{Shnir} A. Shnirelman, {\em On the analyticity of particle trajectories in the ideal
incompressible fluid}, arXiv:1205.5837v

\bibitem{Shubin} M. Shubin, Pseudodifferential operators and spectral theory,
 Second edition, Springer-Verlag, Berlin, 2001

\bibitem{Wol} W. Wolibner, {\em Un theor\`eme sur l'existence du mouvement plan d'un fluide parfait, homog\`ene, 
incompressible, pendant un temps infiniment long}, Math. Z., $\bf 37$(1933), no. 1, 698-726

\end{thebibliography}
\end{document}